\documentclass{siamart190516}
\usepackage{lmodern}
\usepackage[T1]{fontenc}
\usepackage{enumitem}
\usepackage{array}
\usepackage{verbatim}
\usepackage{float}
\usepackage{multirow}
\usepackage{amsmath}
\usepackage{amssymb}
\usepackage{bbding}
\makeatletter

\theoremstyle{plain}
\newtheorem{thm}{\protect\theoremname}
\theoremstyle{plain}
\newtheorem{prop}[thm]{\protect\propositionname}
\theoremstyle{plain}
\newtheorem{lem}[thm]{\protect\lemmaname}

\allowdisplaybreaks
\usepackage[title,toc]{appendix}
\usepackage{url}
\usepackage{enumerate}

\usepackage[nocompress]{cite}
\usepackage{xcolor}
\definecolor{green}{rgb}{0,0.6,0.0} % Darken the default green

\providecommand{\lemmaname}{Lemma}
\providecommand{\propositionname}{Proposition}
\providecommand{\theoremname}{Theorem}

\providecommand{\lemmaname}{Lemma}
\providecommand{\propositionname}{Proposition}
\providecommand{\theoremname}{Theorem}

\makeatother

\providecommand{\lemmaname}{Lemma}
\providecommand{\propositionname}{Proposition}
\providecommand{\theoremname}{Theorem}

\begin{document}
\global\long\def\limn{\lim_{n\to\infty}}%
\global\long\def\rn{\mathbb{R}^{n}}%
\global\long\def\r{\mathbb{R}}%
\global\long\def\R{\mathbb{R}}%
\global\long\def\n{\mathbb{N}}%
\global\long\def\pt{\mathbb{\partial}}%
\global\long\def\tx{\tilde{x}}%
\global\long\def\tx{\tilde{y}}%
\global\long\def\lam{\lambda}%
\global\long\def\cConv{\overline{\text{Conv}}({\cal Z})}%
\global\long\def\argmin{\operatorname*{argmin}}%
\global\long\def\Argmin{\operatorname*{Argmin}}%
\global\long\def\argmax{\operatorname*{argmax}}%
\global\long\def\Argmax{\operatorname*{Argmax}}%
\global\long\def\dom{\operatorname{dom}}%
\global\long\def\lsc{\operatorname{lsc}}%
\global\long\def\cl{\operatorname{cl}}%
\global\long\def\intr{\operatorname{int}}%
\global\long\def\ri{\operatorname{ri}}%
\global\long\def\inner#1#2{\left\langle #1,#2\right\rangle }%

\hyphenation{prob-lem}
\hyphenation{sta-tion-ary}
\hyphenation{imple-men-ta-tions}
\hyphenation{approx-ima-tion}
\hyphenation{approx-imate}
\hyphenation{comp-uta-tional}
\hyphenation{un-con-strained}
\hyphenation{com-plex-ity}
\hyphenation{non-smooth-ness}
\hyphenation{sub-sec-tion}
\hyphenation{pro-po-si-tion}
\hyphenation{sub-sec-tion}
\hyphenation{var-iables}
\hyphenation{approx-im-ate-ly}
\hyphenation{optim-iza-tion}
\hyphenation{in-vol-ving}
\hyphenation{intro-duc-tion}
\hyphenation{pre-sen-ted}
\hyphenation{pro-duct}
\hyphenation{fin-ding}
\hyphenation{con-strained}
\hyphenation{comp-lex-ities}
\hyphenation{illus-trate}
\hyphenation{ass-ump-tions}

% Paper starts here -------------------------------------------------------------

\title{An accelerated inexact proximal point method for solving nonconvex-concave
min-max problems}
\date{\today}
\author{Weiwei Kong\thanks{School of Industrial and Systems     Engineering, Georgia Institute of     Technology, Atlanta, GA, 30332-0205. (E-mails: {\tt        wkong37@gatech.edu} $\&$  {\tt monteiro@isye.gatech.edu}). The works of these authors     were partially supported by ONR Grant N00014-18-1-2077 and NSERC Grant PGSD3-516700-2018.} 
\and Renato D.C. Monteiro\footnotemark[1]}
\maketitle
% This paper presents a quadratic-penalty type method for solving linearly-constrained
% composite nonconvex-concave min-max problems. The method consists
% of solving a sequence of penalty subproblems which, due to the min-max
% structure of the problem, are potentially nonsmooth but can be approximated
% by smooth composite nonconvex minimization problems. Each of these
% penalty subproblems is then solved by applying an accelerated inexact
% proximal point method to its corresponding smooth composite nonconvex
% approximation. Iteration complexity bounds for obtaining approximate
% stationary points of the linearly-constrained composite nonconvex-concave
% min-max problem are also established. Finally, numerical results are
% given to demonstrate the efficiency of the proposed method.
\begin{abstract}
This paper presents smoothing schemes for obtaining approximate stationary points of unconstrained or linearly-constrained
composite nonconvex-concave min-max (and hence nonsmooth) problems by applying well-known algorithms to composite smooth approximations of the original problems. 
More specifically, in the unconstrained (resp. constrained) case, approximate stationary points of the original problem are obtained by applying, to its composite smooth approximation, an accelerated inexact proximal point (resp. quadratic penalty) method presented in a previous paper by the authors. Iteration complexity bounds for both 
smoothing schemes are also established. Finally, numerical results are
given to demonstrate the efficiency of the unconstrained smoothing scheme.
\end{abstract}

% REQUIRED

% REQUIRED

\section{Introduction\label{sec:intro}}

The first goal of this paper is to present and study the complexity
of an accelerated inexact proximal point smoothing (AIPP-S) scheme for
 finding approximate stationary points of the (potentially nonsmooth) min-max composite
nonconvex optimization (CNO) problem 
\begin{equation}
\min_{x \in X} \left\{\hat p(x) :=  p(x) + h(x) \right\} \label{eq:intro_prb}
\end{equation}
where $h$ is a proper lower-semicontinuous convex function, $X$ is a nonempty convex set, and $p$ is a max function given by  
\begin{equation}
p(x) := \max_{y\in Y}\,\Phi(x,y) \quad \forall x \in X, \label{eq:p_def}
\end{equation}
for some nonempty compact convex set $Y$ 
and function $\Phi$ which,
for some scalar $m>0$ and open set $\Omega \supseteq X$, is such that:
(i) $\Phi$ is continuous on $\Omega \times Y$;
(ii) the function $-\Phi(x,\cdot): Y \mapsto \r$ is
lower-semicontinuous and convex
for every $x \in X$;
and (ii) for every $y\in Y$, 
the function $\Phi(\cdot,y) + m\|\cdot\|^2/2$ is convex,
differentiable, and
its gradient is Lipschitz continuous on $X \times Y$.
Here, the objective function is the
sum of a convex function $h$ and the pointwise supremum of (possibly nonconvex) differentiable functions which is
generally a (possibly nonconvex) nonsmooth function.

When $Y$ is a singleton, the max term in \eqref{eq:intro_prb} becomes
smooth and \eqref{eq:intro_prb} reduces to a smooth CNO problem for
which many algorithms have been developed in the literature. In particular,
accelerated inexact proximal points (AIPP) methods, i.e. methods which
use an accelerated composite gradient variant to approximately solve
the generated sequence of prox subproblems, have been developed for
it (see, for example, \cite{WJRproxmet1,Aaronetal2017}). When $Y$
is not a singleton, \eqref{eq:intro_prb} can no longer be directly
solved by an AIPP method due to the nonsmoothness of the max term.
The AIPP-S scheme developed in this paper is instead based on a
perturbed version of \eqref{eq:intro_prb} in which the max term in
\eqref{eq:intro_prb} is replaced by a smooth approximation and the
resulting smooth CNO problem is solved by an AIPP method.

Throughout our presentation, it is assumed that efficient oracles for
evaluating the quantities $\Phi(x,y)$, $\nabla_x \Phi(x,y)$, and $h(x)$ and for obtaining exact solutions of the problems 
\begin{equation}
\min_{x \in X} \left\{\lam h(x) + \frac{1}{2}\|x-x_0\|^2\right\}, \quad \max_{y \in Y} \left\{\lam\Phi(x_0,y)-\frac{1}{2}\|y-y_0\|^2\right\} \label{eq:prox_oracles}
\end{equation}
for any $(x_0, y_0)$ and $\lam>0$, are available.  
% in our proposed algorithms, the total number of calls to these oracles is on the same order of magnitude as the number of calls for the oracles of 
% it is also assumed that efficient oracles for evaluating , are also available. 
Throughout this paper, the terminology ``oracle call''
is used to refer to a collection
of the above oracles of
size ${\cal O}(1)$ where each of them appears at least once.
We refer to the computation of the solution of the first problem above as
a $h$-resolvent evaluation. In this manner, the computation of the solution of the second one is a $[-\Phi(x_0,\cdot)]$-resolvent evaluation.
% evaluations of at most one of each of the aforementioned oracles

We first develop an AIPP-S scheme that obtains a stationary point based on a primal-dual formulation of \eqref{eq:intro_prb}.
More specifically, given a tolerance pair $(\rho_{x},\rho_{y})\in\r_{++}^{2}$,
it is shown that an instance of this scheme obtains a quadruple $(\bar{u},\bar{v},\bar{x},\bar{y})$
such that 
\begin{equation}
\left(\begin{array}{c}
\bar{u}\\
\bar{v}
\end{array}\right)\in\left(\begin{array}{c}
\nabla_{x}\Phi(\bar{x},\bar{y})\\
0
\end{array}\right)+\left(\begin{array}{c}
\partial h(\bar{x})\\
\pt\left[-\Phi(\bar{x},\cdot)\right](\bar{y})
\end{array}\right),\quad\|\bar{u}\|\leq\rho_{x},\quad\|\bar{v}\|\leq\rho_{y}\label{eq:sp_approx_sol}
\end{equation}
in ${\cal O}(\rho_{x}^{-2}\rho_{y}^{-1/2})$ oracle calls, where $\pt \phi(z)$ is the subdifferential of a convex function $\phi$ at a point $z$ (see \eqref{eq:epsubdiff} with $\varepsilon=0$). We then show that another instance of this scheme can obtain an  approximate stationary point
based on the directional derivative of $\hat{p}$. More specifically,
given a tolerance pair $\delta>0$, it is shown that this instance computes
a point $x\in X$ such that 
\begin{equation}
\exists \hat{x}\in X \text{ s.t. } \inf_{\|d\|\leq1}\hat{p}'(\hat{x};d)\geq-\delta,\quad\|\hat{x}-x\|\leq\delta,\label{eq:dd_approx_sol}
\end{equation}
%In this context, we show that this AIPP variant obtains such a point
in ${\cal O}(\delta^{-3})$ oracle calls, where $\hat p'(x;d)$ is the directional derivative of $\hat p$ at the point $x$ along the direction $d$ (see \eqref{eq:ddir_def}).

The second goal of this paper is to develop a quadratic penalty AIPP-S
(QP-AIPP-S) scheme to obtain approximate stationary points of a linearly constrained version
of \eqref{eq:intro_prb}, namely 
\begin{align}
\min_{x\in X}\,\left\{p(x) + h(x):{\cal A}x=b\right\}  \label{eq:intro_lin_prb}
\end{align}
where ${p}$ is as in \eqref{eq:p_def}, ${\cal A}$ is a linear
operator, and $b$ is in the range of ${\cal A}$. The scheme is a
penalty-type method which approximately solves a sequence of penalty
subproblems of the form 
\begin{equation}
\min_{x \in X}\left\{p(x) + h(x) +\frac{c}{2}\|{\cal A}x-b\|^{2}\right\} \label{eq:qp_penalty_prb}
\end{equation}
for an increasing sequence of positive penalty parameters $c$. Similar
to the approach used for the first goal of this paper, the method
considers a perturbed variant of \eqref{eq:qp_penalty_prb} in which
the objective function is replaced by a smooth approximation and the
resulting problem is solved by the quadratic-penalty AIPP (QP-AIPP)
method proposed in \cite{WJRproxmet1}. For a given tolerance triple
$(\rho_{x},\rho_{y},\eta)\in\r_{++}^{3}$, it is shown that the method
computes a quintuple $(\bar{u},\bar{v},\bar{x},\bar{y},\bar{r})$
satisfying 
\begin{equation}
\begin{array}{c}
\left(\begin{array}{c}
\bar{u}\\
\bar{v}
\end{array}\right)\in\left(\begin{array}{c}
\nabla_{x}\Phi(\bar{x},\bar{y})+{\cal A}^{*}\bar{r}\\
0
\end{array}\right)+\left(\begin{array}{c}
\partial h(\bar{x})\\
\pt\left[-\Phi(\bar{x},\cdot)\right](\bar{y})
\end{array}\right),\\
\\
\|\bar{u}\|\leq\rho_{x},\quad\|\bar{v}\|\leq\rho_{y},\quad\|{\cal A}\bar{x}-b\|\leq\eta.
\end{array}\label{eq:lc_sp_approx_sol}
\end{equation}
in ${\cal O}(\rho_{x}^{-2}\rho_{y}^{-1/2}+\rho_{x}^{-2}\eta^{-1})$ oracle calls.

Finally, it is worth mentioning that all of the above complexities are obtained
under the mild assumption that the optimal value in each of the respective
optimization problems, namely \eqref{eq:intro_prb} and \eqref{eq:intro_lin_prb}
is bounded below. Moreover, it is neither assumed that $X$ be
bounded nor that \eqref{eq:intro_prb} or \eqref{eq:intro_lin_prb}
has an optimal solution.\\

\emph{Related Works.} Since the case when $\Phi(\cdot,\cdot)$ in
\eqref{eq:intro_prb} is convex-concave has been well-studied in the
literature (see, for example, \cite{arrow1958studies,nemirovski1978cesari,Nem05-1,OliverMonteiro,he2016accelerated,Rockafellar70,nesterov2005smooth}),
we will make no more mention of it here. Instead, we will focus on
papers that consider \eqref{eq:intro_prb} where $\Phi(\cdot,y)$
is differentiable and nonconvex for every $y\in Y$ and there are mild
conditions on $\Phi(x,\cdot)$ for every $x\in X$.

Letting $\delta_C$ denote the indicator function of a closed convex set $C\subseteq\cal X$ (see Subsection~\ref{subsec:notation_defs}), $\overline{\rm Conv} ({\cal X})$ denote the set of proper lower semicontinuous convex functions on $\cal X$, and $\rho :=\min\{\rho_x,\rho_y\}$,  Tables~\ref{tab:comp1} and \ref{tab:comp2} compare the assumptions and
iteration complexities obtained in this work with corresponding
ones derived in the earlier papers \cite{rafique2018non,2019Nouiehed_Lee} and the subsequent works \cite{lin2020near,ostrovskii2020efficient,thekumparampil2019efficient}. It is worth mentioning that the above works consider termination conditions that are slightly different than the ones in this paper. It is shown in Subsection~\ref{subsec:prelim_asmp} that they are actually equivalent  to the ones in this paper up to multiplicative constants that are independent of the tolerances, i.e., $\rho_x$, $\rho_y$, $\delta$.

\begin{table}[ht]
\begin{centering}
\begin{tabular}{|>{\raggedright}p{2.4cm}|>{\raggedright}p{2.8cm}|>{\centering}p{1.2cm}|>{\centering}p{1.1cm}|>{\centering}p{1.1cm}|>{\centering}p{1.8cm}|}
\hline 
\multirow{2}{2.4cm}{\centering{}{\footnotesize{}Algorithm}} & \multirow{2}{2.8cm}{\centering{}{\footnotesize{}Oracle Complexity}} & \multicolumn{4}{c|}{{\footnotesize{}Use Cases}}\tabularnewline
\cline{3-6} \cline{4-6} \cline{5-6} \cline{6-6} 
 &  & {\footnotesize{}$D_{h}=\infty$} & {\footnotesize{}$h\equiv0$} & {\footnotesize{}$h\equiv\delta_{C}$} & {\footnotesize{}$h\in \overline{\rm Conv} ({\cal X})$}\tabularnewline
\hline 
\centering{}{\scriptsize{}PGSF \cite{2019Nouiehed_Lee}} & \centering{}{\scriptsize{}${\cal O}(\rho^{-3})$} & {\scriptsize{}\XSolidBrush{}} & {\scriptsize{}\Checkmark{}} & {\scriptsize{}\Checkmark{}} & {\scriptsize{}\XSolidBrush{}}\tabularnewline
\centering{}{\scriptsize{}Minimax-PPA \cite{lin2020near}} & \centering{}{\scriptsize{}${\cal O}(\rho^{-2.5}\log^{2}(\rho^{-1}))$} & {\scriptsize{}\XSolidBrush{}} & {\scriptsize{}\Checkmark{}} & {\scriptsize{}\Checkmark{}} & {\scriptsize{}\XSolidBrush{}}\tabularnewline
\centering{}{\scriptsize{}FNE Search \cite{ostrovskii2020efficient}} & \centering{}{\scriptsize{}${\cal O}(\rho_{x}^{-2}\rho_{y}^{-1/2}\log(\rho^{-1}))$} & {\scriptsize{}\Checkmark{}} & {\scriptsize{}\Checkmark{}} & {\scriptsize{}\Checkmark{}} & {\scriptsize{}\XSolidBrush{}}\tabularnewline
\centering{}\textbf{\scriptsize{}AIPP-S} & \centering{}{\scriptsize{}${\cal O}(\rho_{x}^{-2}\rho_{y}^{-1/2})$} & {\scriptsize{}\Checkmark{}} & {\scriptsize{}\Checkmark{}} & {\scriptsize{}\Checkmark{}} & {\scriptsize{}\Checkmark{}}\tabularnewline
\hline 
\end{tabular}
\par\end{centering}
\caption{Comparison of iteration complexities and assumptions under notions equivalent to \eqref{eq:sp_approx_sol} with $\rho:=\min\{\rho_{x},\rho_{y}\}$.\label{tab:comp1}}
\end{table}

\begin{table}[ht]
\begin{centering}
\begin{tabular}{|>{\raggedright}p{2.4cm}|>{\raggedright}p{2.8cm}|>{\centering}p{1.2cm}|>{\centering}p{1.1cm}|>{\centering}p{1.1cm}|>{\centering}p{1.8cm}|}
\hline 
\multirow{2}{2.4cm}{\centering{}{\footnotesize{}Algorithm}} & \multirow{2}{2.8cm}{\centering{}{\footnotesize{}Oracle Complexity}} & \multicolumn{4}{c|}{{\footnotesize{}Use Cases}}\tabularnewline
\cline{3-6} \cline{4-6} \cline{5-6} \cline{6-6} 
 &  & {\footnotesize{}$D_{h}=\infty$} & {\footnotesize{}$h\equiv0$} & {\footnotesize{}$h\equiv\delta_{C}$} & {\footnotesize{}$h\in \overline{\rm Conv} ({\cal X})$}\tabularnewline
\hline 
\centering{}{\scriptsize{}PG-SVRG \cite{rafique2018non}} & \centering{}{\scriptsize{}${\cal O}(\delta^{-6}\log\delta^{-1})$} & {\scriptsize{}\XSolidBrush{}} & {\scriptsize{}\Checkmark{}} & {\scriptsize{}\Checkmark{}} & {\scriptsize{}\Checkmark{}}\tabularnewline
\centering{}{\scriptsize{}Minimax-PPA \cite{lin2020near}} & \centering{}{\scriptsize{}${\cal O}(\delta^{-3}\log^{2}(\delta^{-1}))$} & {\scriptsize{}\XSolidBrush{}} & {\scriptsize{}\Checkmark{}} & {\scriptsize{}\Checkmark{}} & {\scriptsize{}\XSolidBrush{}}\tabularnewline
\centering{}{\scriptsize{}Prox-DIAG \cite{thekumparampil2019efficient}} & \centering{}{\scriptsize{}${\cal O}(\delta^{-3}\log^{2}(\delta^{-1}))$} & {\scriptsize{}\Checkmark{}} & {\scriptsize{}\Checkmark{}} & {\scriptsize{}\XSolidBrush{}} & {\scriptsize{}\XSolidBrush{}}\tabularnewline
\centering{}\textbf{\scriptsize{}AIPP-S} & \centering{}{\scriptsize{}${\cal O}(\delta^{-3})$} & {\scriptsize{}\Checkmark{}} & {\scriptsize{}\Checkmark{}} & {\scriptsize{}\Checkmark{}} & {\scriptsize{}\Checkmark{}}\tabularnewline
\hline 
\end{tabular}
\par\end{centering}
\caption{Comparison of iteration complexities and assumptions under notions equivalent to \eqref{eq:dd_approx_sol}.\label{tab:comp2}}
\end{table}

To the best of our knowledge, this work is the first one to analyze the complexity of a smoothing scheme for finding approximate stationary points of \eqref{eq:intro_lin_prb}.
\\

\emph{Organization of the paper.}  Subsection~\ref{subsec:notation_defs}
presents notation and some basic definitions that are used in this
paper. Subsection~\ref{subsec:motivating_apps} presents several motivating applications that are of the form in \eqref{eq:intro_prb}. Section~\ref{sec:prelim} is divided into two subsections. The
first one precisely states the assumptions underlying problem \eqref{eq:intro_prb} and discusses four notions of stationary points.
The second one presents a smooth approximation of the function $p$ in \eqref{eq:intro_prb}. Section~\ref{sec:mm_opt}
is divided into two subsections. The first one reviews the AIPP method in \cite{WJRproxmet1} and its iteration complexity. The second one presents
the AIPP-S scheme its iteration complexities for finding approximate stationary points as in \eqref{eq:sp_approx_sol} and \eqref{eq:dd_approx_sol}. Section~\ref{subsec:lc_mm_opt} is also divided into two subsections. The first one reviews the QP-AIPP method in \cite{WJRproxmet1} and its iteration complexity. The second one presents
the QP-AIPP-S scheme its iteration complexity for finding points satisfying \eqref{eq:lc_sp_approx_sol}. Section~\ref{sec:numerical} presents some computational results. Section~\ref{sec:concl_remarks} gives some concluding remarks. Finally, several appendices at the end of this paper contain proofs of technical results needed in our presentation.

\subsection{Notation and basic definitions\label{subsec:notation_defs}}

This subsection provides some basic notation and definitions.

The set of real numbers is denoted by $\r$. The set of non-negative
real numbers and the set of positive real numbers is denoted by $\r_{+}$
and $\r_{++}$ respectively. The set of natural numbers is denoted
by $\n$. For $t>0$, define $\log_{1}^{+}(t):=\max\{1,\log(t)\}$.
Let $\rn$ denote a real--valued $n$--dimensional Euclidean space
with standard norm $\|\cdot\|$. Given a linear operator $A:\rn\mapsto\r^{p}$,
the operator norm of $A$ is denoted by $\|A\|:=\sup\{\|Az\|/\|z\|:z\in\rn,z\neq0\}$.

The following notation and definitions are for a general complete
inner product space ${\cal Z}$, whose inner product and its associated
induced norm are denoted by $\left\langle \cdot,\cdot\right\rangle $
and $\|\cdot\|$ respectively. Let $\psi:{\cal {\cal Z}}\mapsto(-\infty,\infty]$
be given. The effective domain of $\psi$ is denoted as $\dom\psi:=\{z\in{\cal {\cal Z}}:\psi(z)<\infty\}$
and $\psi$ is said to be proper if $\dom\psi\neq\emptyset$. 
The set of proper, lower
semi-continuous, convex functions $\psi:{\cal {\cal Z}}\mapsto(-\infty,\infty]$
is denoted by $\cConv$. Moreover, for a convex set $Z\subseteq {\cal Z}$, we denote $\overline{\rm Conv}(Z)$ to be set of functions in $\cConv$ whose effective domain is equal to $Z$.
For
$\varepsilon\geq0$, the $\varepsilon$\emph{-subdifferential }of
$\psi\in \cConv$ at $z\in\dom\psi$ is denoted by 
\begin{equation}
\pt_{\varepsilon}\psi(z):=\left\{ w\in\rn:\psi(z')\geq\psi(z)+\left\langle w,z'-z\right\rangle -\varepsilon,\forall z'\in{\cal {\cal Z}}\right\} ,\label{eq:epsubdiff}
\end{equation}
and we denote $\pt\psi\equiv\pt_{0}\psi$. The \emph{directional derivative} of $\psi$
at $z\in{\cal Z}$ in the direction $d\in{\cal Z}$ is denoted by
\begin{equation}
\psi'(z;d):=\lim_{t\to0}\frac{\psi(z+td)-\psi(z)}{t}. \label{eq:ddir_def}
\end{equation}
It is well-known that if $\psi$ is differentiable at $z\in\dom\psi$,
then for a given direction $d\in{\cal Z}$ we have $\psi'(z;d)=\left\langle \nabla\psi(z),d\right\rangle $.

For a given $Z\subseteq{\cal Z}$, the indicator function of $Z$,
denoted by $\delta_{Z}$, is defined as $\delta_{Z}(z)=0$ if $z\in Z$
and $\delta_{Z}(z)=\infty$ if $z\notin Z$. Moreover, the closure, interior, and relative interior of $Z$ are denoted by $\cl Z$, $\intr Z$, and $\ri Z$, respectively. The support function of $Z$ at a point $z$ is denoted by $\sigma_Z(z):=\sup_{z'\in Z} \inner{z}{z'}$. 

\subsection{\label{subsec:motivating_apps}Motivating applications}

This subsection lists motivating applications that
are of the form in \eqref{eq:intro_prb}. In Section~\ref{sec:numerical},
we examine the performance of our proposed smoothing scheme on some special instances of these applications.

\subsubsection{Maximum of a finite number of nonconvex functions}

Given a family of functions $\{f_{i}\}_{i=1}^{k}$ that
are continuously differentiable everywhere with Lipschitz continuous gradients and a closed convex set $C\subseteq\rn$.
The problem of interest is the minimization of $\max_{1\leq i\leq k}f_{i}$ over the
set $C$, i.e., 
\[
\min_{x\in C} \max_{1\leq i\leq k} f_{i}(x),
\]
which is clearly an instance of \eqref{eq:intro_prb} where
$ Y=\{ y\in\r_{+}^{k}:\sum_{i=1}^{k}y_{i}=1\}$, $\Phi(x,y)=\sum_{i=1}^{k}y_{i}f_{i}(x)$, and $h(x)=\delta_{C}(x)$.

\subsubsection{Robust regression}

Given a set of observations $\sigma:=\{\sigma_{i}\}_{i=1}^{n}$ and
a compact convex set $\Theta\in\r^{k}$, let $\{\ell_{\theta}(\cdot|\sigma)\}_{\theta\in\Theta}$
be a family of nonconvex loss functions in which: (i) $\ell_{\theta}(x|\sigma)$
is concave in $\theta$ for every $x\in\rn$; and (ii) $\ell_{\theta}(x|\sigma)$
is continuously differentiable in $x$ with Lipschitz continuous gradient
for every $\theta\in\Theta$. The problem of interest is to minimize
the worst-case loss in $\Theta$, i.e., 
\[
\min_{x\in\rn}\max_{\theta\in\Theta}\ell_{\theta}(x|\sigma),
\]
which is clearly an instance of \eqref{eq:intro_prb},  where 
$Y=\Theta$, $\Phi(x,y)=\ell_{y}(x|\sigma)$, and $h(x)=0$.

\subsubsection{Min-max games with an adversary}

Let $\{{\cal U}_{j}(x_{1},...,x_{k},y)\}_{j=1}^{k}$
be a set of utility functions 
% with Lipschitz continuous
% gradients, where $x_{j}$ represents the inputs of the $j^{{\rm th}}$
% user and $y$ represents inputs of the adversary. Moreover, 
% suppose that: 
in which: (i) ${\cal U}_{j}$ is nonconvex and continuously differentiable in its first $k$ arguments,
but concave in its last argument; (ii) $ \nabla_{x_i} {\cal U}_{j}(x_{1},...,x_{k},y)$ is Lipschitz continuous for every $1\leq i \leq k$.
Given input constraint sets $\{B_{i}\}_{i=1}^{k}$
and $B_{y}$, the problem of interest is to maximize the total utility of the players (indices $1$ to $k$)
given that the adversary (index $k+1$) seeks to maximize his own utility, i.e.,
\begin{align*}
\min_{x_{1},...,x_{k}}\max_{y}\left\{ -\sum_{i=1}^{k}{\cal U}_{j}(x_{1},...,x_{k},y):x_{i}\in B_{i},i=0,...,k\right\} ,
\end{align*}
which is clearly an instance of \eqref{eq:intro_prb} where $x=(x_{1},...,x_{k})$, $Y=B_{y}$, $\Phi(x,y)=-\sum_{i=1}^{k}{\cal U}_{j}(x_{1},...,x_{k},y)$, and $h(x)=\delta_{B_{1}\times...\times B_{k}}(x)$.

\section{Preliminaries\label{sec:prelim}}

This section presents some preliminary material and is divided into two subsections. The first one precisely describes the assumptions and various notions of stationary points for problem \eqref{eq:intro_prb} and briefly compares two approaches for obtaining them. The second one presents a smooth approximation of the max function $p$ in \eqref{eq:intro_prb} and some of its properties. 

\subsection{Assumptions and notions of stationary points\label{subsec:prelim_asmp}}

This subsection describes the assumptions and four notions of stationary points for
for problem \eqref{eq:intro_prb}.
It is worth mentioning that the complexities of the smoothing scheme of Section~\ref{sec:mm_opt} are
presented with respect
to two of these notions.
In order to understand how
these results can be translated to the other two
alternative notions,
which have been used
in a few papers dealing
with problem \eqref{eq:intro_prb},
% of stationary points,
we present a few results
discussing some useful
relations between all these notions.

% The other two are alternative notions of stationary
% points which have been considered in the
% literature, and we give a discussion of 
% the relationships between
% all of the aforementioned notions.
% \eqref{eq:intro_prb} which will be the reference for
% the complexity results which we base our smoothing scheme on.
% that will be the sole focus of this paper.

Throughout our presentation, we let ${\cal X}$ and ${\cal Y}$ be finite dimensional inner product spaces. 
We also make the following assumptions on problem \eqref{eq:intro_prb}:
\begin{itemize}
\item[(A0)] $X \subset {\cal X}$ and $Y \subset {\cal Y}$ are nonempty convex sets, and
$Y$ is also compact;
\item[(A1)]
there exists an open set $\Omega \supseteq X$ such that
$\Phi(\cdot,\cdot)$ is finite and continuous on
$\Omega \times Y$; moreover, $\nabla_x \Phi(x,y)$ exists and is continuous at every $(x,y)\in\Omega \times Y$;
\item[(A2)] 
$h\in \overline{\rm Conv}(X)$ and
$-\Phi(x, \cdot)\in\overline{\rm Conv}(Y)$ for every $x\in \Omega$;
% \item[(A2)]  $\nabla_x \Phi(x,y)$ exists for every $(x,y)\in\Omega \times Y$ and both $\Phi(\cdot,\cdot)$ and $\nabla_x \Phi(\cdot,\cdot)$ are continuous on $\Omega \times Y$; 
\item[(A3)] there exist scalars $(L_{x},L_{y})\in\r_{++}^{2},$ and $m\in(0,L_{x}]$
such that 
\begin{align}
\Phi(x,y)-\left[\Phi(x',y)+\left\langle \nabla_{x}\Phi(x',y),x-x'\right\rangle \right]\geq-\frac{m}{2}\|x-x'\|^{2},\label{eq:sp_lower_curv}\\
\|\nabla_{x}\Phi(x,y)-\nabla_{x}\Phi(x',y')\|\leq L_{x}\|x-x'\|+L_{y}\|y-y'\|,\label{eq:M_L_xy}
\end{align}
for every $x,x'\in X$ and $y,y'\in Y$;
% \item[(A4)] $Y$ is closed and has a finite diameter $D_{y}:=\sup\left\{ \|y-y'\|:y,y'\in Y\right\}$; 
\item[(A4)] $\hat{p}_{*} := \inf_{x\in X} \hat p(x)$ is finite, where $\hat p$ is as in \eqref{eq:intro_prb};
\end{itemize}

We make three remarks about the above assumptions. 
First, it is well-known
that condition \eqref{eq:M_L_xy} implies that 
\begin{equation}
\Phi(x',y)-\left[\Phi(x,y)+\left\langle \nabla_{x}\Phi(x,y),x'-x\right\rangle \right]\leq\frac{L_{x}}{2}\|x'-x\|^{2},\label{eq:sp_upper_curv}
\end{equation}
for every $(x',x,y)\in X\times X\times Y$.
Second, functions satisfying the lower curvature condition in \eqref{eq:sp_lower_curv} are often referred to as weakly convex functions (see, for example, \cite{davis2019stochastic, duchi2018stochastic,davis2018subgradient,drusvyatskiy2019efficiency}).
Third, the aforementioned weak convexity condition implies that, for any $y\in Y$,
the function $\Phi(\cdot,y)+m\|\cdot\|^{2}/2$ is convex, and hence
$p+m\|\cdot\|^{2}/2$ is as well. Note that while $\hat{p}$ is generally
nonconvex and nonsmooth, it has the nice property that $\hat{p}+m\|\cdot\|^{2}/2$
is convex.

We now discuss two stationarity conditions of \eqref{eq:intro_prb} under assumptions (A0)--(A3). First, denoting \begin{equation} \label{eq:Phi_hat_def}
\hat \Phi(x,y) :=\Phi(x,y)+h(x) \quad \forall (x,y)\in X\times Y,
\end{equation}
it is well-known that \eqref{eq:intro_prb} is related to the saddle-point problem which consists of finding a pair $(x^{*},y^{*})\in X \times Y$
such that 
\begin{equation}
\hat{\Phi}(x^{*},y)\leq\hat{\Phi}(x^{*},y^{*})\leq\hat{\Phi}(x,y^{*}),\label{eq:saddle_point}
\end{equation}
for every $(x,y)\in X \times Y$. More specifically, $(x^{*},y^{*})$ satisfies \eqref{eq:saddle_point}
if and only if $x^{*}$ is an optimal solution of \eqref{eq:intro_prb},
$y^{*}$ is an optimal solution of the dual of \eqref{eq:intro_prb},
and there is no duality gap between the two problems. Using the composite
structure described above for $\hat{\Phi}$, it can be shown that a necessary condition for \eqref{eq:saddle_point} to hold is that $(x^{*},y^{*})$ satisfy the stationarity condition
\begin{equation}
\left(\begin{array}{c}
0\\
0
\end{array}\right)\in\left(\begin{array}{c}
\nabla_{x}\Phi(x^{*},y^{*})\\
0
\end{array}\right)+\left(\begin{array}{c}
\partial h(x^{*})\\
\pt\left[-\Phi(x^{*},\cdot)\right](y^{*})
\end{array}\right).\label{eq:crit_conv_incl}
\end{equation}
When $m=0$, the above condition also becomes sufficient for \eqref{eq:saddle_point} to hold. Second, it can be shown that $p'(x^*;d)$ is well-defined for every
$d \in {\cal X}$ and that a necessary condition for $x^{*}\in X$ to be a local minimum of \eqref{eq:intro_prb} is that it satisfies the stationarity condition
\begin{equation}
\inf_{\|d\|\leq1}\hat{p}'(x^{*};d)\geq0.\label{eq:ddir_conv}
\end{equation}
When $m=0$, the above condition also becomes sufficient for $x^*$ to be a global minimum of \eqref{eq:intro_prb}. Moreover, in view of Lemma~\ref{lem:approx_unconstr} in Appendix~\ref{app:statn_notions} with $(\bar{u},\bar{v},\bar{x},\bar{y})=(0,0,x^*,y^*)$, it follows that $x^*$ satisfies \eqref{eq:ddir_conv} if and only if there exists
$y^* \in Y$ such that $(x^*,y^*)$ satisfies \eqref{eq:crit_conv_incl}.
%  Our second notion of stationarity arises from viewing \eqref{eq:intro_prb} as a general nonsmooth CNO problem. Under this setting, if $\hat p$ is convex then $x^*$ is an optimal solution of \eqref{eq:intro_prb} if and only if 
% \begin{equation}
% \inf_{\|d\|\leq1}\hat{p}'(x^{*};d)\geq0.\label{eq:ddir_conv}
% \end{equation}
% if $\hat p$ is nonconvex, then
%  \eqref{eq:crit_conv_incl} is only a necessary stationarity condition for \eqref{eq:saddle_point} to hold for every $(x,y)\in X \times Y$.
%  When $\hat p$ is convex, it is also
%  well-known that
% If $\hat{p}$ is nonconvex then it can be shown that
%  but
% \eqref{eq:ddir_conv} is only a.
% Finally, if $(x^*, y^*, r^*)$ satisfies \eqref{eq:lc_crit_conv_incl} and $\cal F$ is as in assumption (A6), then \eqref{eq:saddle_point} holds for every $(x,y)\in {\cal F}\times Y$, but the reverse implication is generally not true. 

Note that finding points that satisfy %\eqref{eq:saddle_point},
\eqref{eq:crit_conv_incl} or \eqref{eq:ddir_conv} exactly
is generally a difficult task. Hence, in this section and the next one, we only consider approximate versions of \eqref{eq:crit_conv_incl} or \eqref{eq:ddir_conv}, which are \eqref{eq:sp_approx_sol} and \eqref{eq:dd_approx_sol}, respectively. For ease of future reference, we say that: 
\begin{itemize}
\item[(i)] a quadruple $(\bar{u},\bar{v},\bar{x},\bar{y})$
is a \textbf{$(\rho_{x},\rho_{y})$--primal-dual stationary point} of
\eqref{eq:intro_prb} if it satisfies \eqref{eq:sp_approx_sol};
\item[(ii)] a point $\hat{x}$ is a \textbf{$\delta$--directional stationary
    point} of \eqref{eq:intro_prb} if it satisfies the first inequality
    in \eqref{eq:dd_approx_sol}.
\end{itemize}
It is worth mentioning that \eqref{eq:dd_approx_sol} is generally hard to verify
for a given point $x\in X$. This is primarily because the definition
requires us to check an infinite number of directional derivatives
for a (potentially) nonsmooth function at points $\hat{x}$ near $\bar{x}$.
In contrast, the definition of an approximate primal-dual stationary point
is generally easier to verify because the quantities $\|\bar{u}\|^ {}$
and $\|\bar{v}\|^ {}$ can be measured directly, and the inclusions
in \eqref{eq:sp_approx_sol} are easy to verify when the prox oracles for $h$ and $\Phi(x,\cdot)$,
for every $x\in X$, are readily available.

% We now state some results that relate \eqref{eq:sp_approx_sol} with \eqref{eq:dd_approx_sol}. The first result, whose proof can be found in Appendix~\ref{app:statn_notions}, describes how our two notions of approximate stationary points of \eqref{eq:intro_prb} are related.
% \begin{prop} \label{lem:approx_unconstr}
% Let a pair $(\Phi,h)$ satisfying assumptions (A1)--(A3) of Subsection~\ref{subsec:prelim_asmp} and 
% $(\bar x,\bar{v})\in X \times {\cal Y}$ be given.
% Then, $\bar{x}$ is a $\delta$-directional stationary point
% of the perturbed min-max problem
% \begin{equation}
% \min_{x\in X}\max_{y\in Y}\left[\Phi(x,y)+\left\langle \bar{v},y\right\rangle +h(x)\right],\label{eq:perturbed_mm}
% \end{equation} 
% if and only if there exists $(\bar{u},\bar{y})\in {\cal X} \times Y$
% such that $(\bar{u},\bar{v},\bar{x},\bar{y})$ is a
% $(\delta,\|\bar{v}\|)$-primal-dual stationary point of \eqref{eq:intro_prb}.
% \end{prop}

The next result, whose proof is given in Appendix~\ref{app:statn_notions}, shows that a $(\rho_x, \rho_y)$--primal-dual stationary point, for small enough $\rho_x$ and $\rho_y$, yields a point $x$ satisfying \eqref{eq:dd_approx_sol}.
Its statement makes use of the
diameter of $Y$ defined as
\begin{equation}
D_y := \sup_{y,y'\in Y} \|y-y'\|. \label{eq:Dy_def}
\end{equation}
\begin{prop}\label{prop:impl1_statn}
If the quadruple $(\bar{u},\bar{v},\bar{x},\bar{y})$ is a $(\rho_x,\rho_y)$--primal-dual stationary point of \eqref{eq:intro_prb},
then there exists a point $\hat{x}\in X$ such that 
\[
\inf_{\|d\|\leq1} \hat{p}'(\hat{x};d)\geq-\rho_x - 2\sqrt{2mD_{y} \rho_y},\quad\|\bar{x}-\hat{x}\|\leq\sqrt{\frac{2D_{y}\rho_y}{m}}.
\]
\end{prop}

The iteration complexities in this paper (see Section~\ref{sec:mm_opt}) are stated with respect to the two notions of stationary points \eqref{eq:sp_approx_sol} and \eqref{eq:dd_approx_sol}.
However, it is worth discussing
below
two other notions of stationary points that are common in the literature as well as some results that relate all four notions.

Given $(\lambda,\varepsilon)\in\r_{++}^2$, a point $x$ is said to be a $(\lambda,\varepsilon)$-prox stationary point of \eqref{eq:intro_prb} if the function $\hat p +\|\cdot\|^2/(2\lambda)$ is strongly convex and
\begin{equation}
    \frac{1}{\lambda}\|x - x_\lambda\| \leq \varepsilon, \quad x_\lambda = \argmin_{u \in {\cal X}} \left\{\hat{P}_\lam(u) := \hat p(u) + \frac{1}{2\lambda}\|u-x\|^2 \right\}. \label{eq:prox_stn_point}
\end{equation}
The above notion is considered, for example, in \cite{rafique2018non,lin2020near,thekumparampil2019efficient}. The result below, whose proof is given in Appendix~\ref{app:statn_notions}, shows how it is related to \eqref{eq:dd_approx_sol}.

\begin{prop} \label{eq:prop_prox_statn}
For any given $\lambda \in (0,1/m)$, the following statements hold:
\begin{itemize}
    \item[(a)] for any
 $\varepsilon>0$,
    if $x \in X$ 
    satisfies
    \eqref{eq:dd_approx_sol} and 
    \begin{equation}
0<\delta\leq \frac{\lambda^3 \varepsilon}{\lambda^2 + 2(1-\lambda m)(1+\lambda)}, \label{eq:spec_delta_bd}
\end{equation}
    then $x$ is
a $(\lambda,\varepsilon)$-prox stationary point;
\item[(b)] for any
$\delta>0$,
if $x \in X$
is a $(\lambda,\varepsilon)$-prox stationary point for some $\varepsilon\leq\delta\cdot\min\{1,1/\lambda\}$,
then $x$ satisfies
\eqref{eq:dd_approx_sol} with $\hat{x}=x_{\lambda}$,
where $x_{\lambda}$ is as in \eqref{eq:prox_stn_point}.
\end{itemize}
\end{prop}
Note that for a fixed $\lam\in(0,1/m)$ such that
$\max\{ \lam^{-1}, (1-\lam m)^{-1}\} = {\cal O}(1)$, the largest $\delta$ in part (a) is ${\cal O}(\varepsilon)$. Similarly, for part (b), if $\lam^{-1}={\cal O}(1)$ then largest $\varepsilon$ in part (b) is ${\cal O}(\delta)$. Combining these two observations, it
follows that \eqref{eq:prox_stn_point} and \eqref{eq:dd_approx_sol} are equivalent (up to a multiplicative factor) under the assumption that
$\delta=\Theta(\varepsilon)$. 

Given $(\rho_x,\rho_y) \in \r_{++}^2$, a pair $(\bar x, \bar y)$ is said to be a $(\rho_x,\rho_y)$-first-order Nash equilibrium point of \eqref{eq:intro_prb} if 
\begin{equation}
\inf_{\|d_x\|\leq 1} {\cal S}_{\bar y}'(\bar x;d_x) \geq -\rho_x, \quad \sup_{\|d_y\|\leq 1} {\cal S}_{\bar x}'(\bar y; d_y) \leq \rho_y, \label{eq:nash_stn_point}
\end{equation}
where ${\cal S}_{\bar y} :=\Phi(\cdot,\bar y) + h(\cdot)$ and ${\cal S}_{\bar x} := \Phi(\bar x, \cdot)$. The above notion is considered, for example, in \cite{2019Nouiehed_Lee,lin2020near,ostrovskii2020efficient}. The next result, whose proof is given in Appendix~\ref{app:statn_notions}, shows that \eqref{eq:nash_stn_point} is equivalent to \eqref{eq:sp_approx_sol}.

\begin{prop} \label{prop:ne_statn_pt}
A pair $(\bar{x},\bar{y})$ is a $(\rho_{x},\rho_{y})$-first-order
Nash equilibrium point if and only if there exists $(\bar{u},\bar{v})\in{\cal X}\times{\cal Y}$
such that $(\bar{u},\bar{v},\bar{x},\bar{y})$ satisfies \eqref{eq:sp_approx_sol}.
\end{prop}

We now end this subsection by briefly discussing some approaches for finding approximate stationary points of \eqref{eq:intro_prb}. One approach is to apply
a proximal descent type method directly to problem \eqref{eq:intro_prb},
but this would lead to subproblems with nonsmooth convex composite
functions. A second approach is based on first applying a smoothing
method to \eqref{eq:intro_prb} and then using a prox-convexifying descent method such as the one in \cite{WJRproxmet1} to solve the
perturbed unconstrained smooth problem. An advantage of the second approach, which
is the one pursued in this paper, is that it generates subproblems with smooth convex composite objective functions.
The next subsection describes one possible way to smooth the  (generally) nonsmooth function $p$ in \eqref{eq:intro_prb}.

\subsection{Smooth approximation\label{subsec:sp_aipp_smooth}}

This subsection presents a smooth approximation of the function $p$ in \eqref{eq:intro_prb}.

For every $\xi>0$, consider the smoothed function $p_{\xi}$ defined by 
\begin{align}
p_{\xi}(x) & :=\max_{y\in Y}\left\{ \Phi_{\xi}(x,y):=\Phi(x,y)-\frac{1}{2\xi}\|y-y_{0}\|^{2}\right\} \quad\forall x\in X,\label{eq:p_xi_def}
\end{align}
for some $y_{0}\in Y$. The following proposition presents the key properties of $p_\xi$ and its related quantities.

% The difference between \eqref{eq:smoothed_intro_prb}
% and \eqref{eq:primal_sp} is that the function $p$ in \eqref{eq:p_def}
% is replaced by the function $p_{\xi}:X\mapsto\r$ which approximates
% $p$.

% In order for this approach to be valid, we need to establish that
% \eqref{eq:smoothed_intro_prb} is a problem that can be solved by the
% AIPP method. As $h\in\text{\ensuremath{\overline{\text{Conv}}}}({\cal X})$,
% it is sufficient to show that $p_{\xi}$ satisfies assumption (P2)
% in Subsection~\ref{subsec:aipp}. This is done in the following results
% which also give additional properties about the functions $p_{\xi}$
% and $\hat{p}_{\xi}$ as in \eqref{eq:p_xi_def} and \eqref{eq:smoothed_intro_prb},
% respectively, and the optimal solution of \eqref{eq:p_xi_def} as
% a function of $x$. 
\begin{prop}
\label{prop:xi_facts} Let $\xi>0$ be given and assume that the function $\Phi$ satisfies conditions (A0)--(A3). Let $p_{\xi}(\cdot)$ and $\Phi_{\xi}(\cdot,\cdot)$ be as defined
in \eqref{eq:p_xi_def} and define 
\begin{equation}
\begin{gathered}
Q_{\xi}:=\xi L_{y}+\sqrt{\xi(L_{x}+m)}, \quad L_{\xi}:=L_{y}Q_{\xi}+L_{x}\leq\left(L_{y}\sqrt{\xi}+\sqrt{L_{x}}\right)^{2}, \\ y_{\xi}(x):=\argmax_{y'\in Y}\Phi_{\xi}(x,y'),
\end{gathered} \label{eq:y_xi_def}
\end{equation}
for every $x \in X$. Then, the following properties hold: 
\begin{itemize}
\item[(a)] $y_{\xi}(\cdot)$ is $Q_{\xi}$--Lipschitz continuous on $X$;
\item[(b)] $p_{\xi}(\cdot)$ is continuously differentiable on $X$ and $\nabla p_{\xi}(x)=\nabla_{x}\Phi(x,y_{\xi}(x))$
for every $x\in X$; 
\item[(c)] $\nabla p_{\xi}(\cdot)$ is $L_{\xi}$--Lipschitz continuous on $X$;
\item[(d)] for every $x,x'\in X$, we have 
\begin{equation}
p_{\xi}(x)-\left[p_{\xi}(x')+\left\langle \nabla p_{\xi}(x'),x-x'\right\rangle \right]\geq-\frac{m}{2}\|x-x'\|^{2};\label{eq:g_xi_lower_curv}
\end{equation}
\end{itemize}
\end{prop}

\begin{proof}
Note that the inequality in \eqref{eq:y_xi_def} follows from (a), the fact
that $m\leq L_{x}$, and the bound 
\[
L_{\xi}=L_{y}\left[\xi L_{y}+\sqrt{\xi(L_{x}+m)}\right]+L_{x}\leq\xi L_{y}^{2}+2\sqrt{\xi L_{x}}+L_{x}=\left(L_{y}\sqrt{\xi}+\sqrt{L_{x}}\right)^{2}.
\]
The other conclusions of (a)--(c) follow from Lemma~\ref{lem:diff_danskin}
and Proposition~\ref{prop:Psi_global_ext} in Appendix~\ref{app:smoothing}
with $(\Psi,q,y)=(\Phi_{\xi},p_{\xi},y_\xi)$. We now show that the conclusion
of (d) is true. Indeed, if we consider \eqref{eq:sp_lower_curv} at
$(y,x')=(y_{\xi}(x'),x')$, the definition of $\Phi_{\xi}$, and use
the definition of $\nabla p_{\xi}$ in (b), then 
\begin{align*}
 & -\frac{m}{2}\|x-x'\|^{2}\leq\Phi(x',y_{\xi}(x))-\left[\Phi(x,y_{\xi}(x))+\left\langle \nabla_{x}\Phi(x,y_{\xi}(x)),x'-x\right\rangle \right]\\
% & =\Phi(x',y_{\xi}(x))-\frac{1}{2\xi}\|y_{\xi}(x)-y_{0}\|^{2}-\\
% & \quad\quad\left[\Phi(x,y_{\xi}(x))-\frac{1}{2\xi}\|y_{\xi}(x)-y_{0}\|^{2}+\left\langle \nabla_{x}\Phi(x,y_{\xi}(x)),x'-x\right\rangle \right]\\
 & =\Phi_{\xi}(x',y_{\xi}(x))-\left[p_{\xi}(x)+\left\langle \nabla p_{\xi}(x),x'-x\right\rangle \right]\leq p_{\xi}(x')-\left[p_{\xi}(x)+\left\langle \nabla p_{\xi}(x),x'-x\right\rangle \right],
\end{align*}
where the last inequality follows from the optimality of $y$. 
\end{proof}

We now make two remarks about the above properties. First,
the Lipschitz constants of $y_\xi$ and $\nabla p_{\xi}$ depend on the value of
$\xi$ while the weak convexity constant $m$ in \eqref{eq:g_xi_lower_curv}
does not. Second, as $\xi \to \infty$, it holds that $p_\xi \to p$  pointwise and $Q_\xi, L_\xi \to \infty$. These remarks are made more precise in the next result.

\begin{lem} \label{lem:smoothing_relation}
For every $\xi>0$, it holds that $-\infty<p(x)-{D_{y}^{2}}/({2\xi})\leq p_{\xi}(x)\leq p(x)$ for every $x\in X$, 
where $D_y$ is as in \eqref{eq:Dy_def}.
\end{lem}

\begin{proof}
The fact that $p(x)>-\infty$ follows immediately from assumption
(A4). To show the other bounds, observe that for every $y_{0}\in Y$,
we have 
\[
\Phi(x,y)+h(x)\geq\Phi(x,y)-\frac{1}{2\xi}\|y-y_{0}\|^{2}+h(x)\geq\Phi(x,y)-\frac{D_{y}^{2}}{2\xi}+h(x)
\]
for every $(x,y)\in X\times Y$. Taking the supremum of the bounds
over $y\in Y$ and using the definitions of $p$ and $p_{\xi}$ yields
the remaining bounds.
\end{proof}

\section{Unconstrained min-max optimization\label{sec:mm_opt}}

This section presents our proposed AIPP-S scheme for solving the min-max CNO problem \eqref{eq:intro_prb} and is divided into two subsections. The first one reviews an AIPP method for solving smooth CNO problems. The second one presents the AIPP-S scheme and its iteration complexity for finding stationary points as in \eqref{eq:sp_approx_sol} and \eqref{eq:dd_approx_sol}.

Before proceeding, we briefly outline the idea of the AIPP-S scheme. The main idea is to apply the AIPP method described in the next subsection to the smooth CNO problem
\begin{equation}
\min_{x \in X}\left\{ \hat{p}_{\xi}(x):=p_{\xi}(x)+h(x)\right\}, \label{eq:smoothed_intro_prb}
\end{equation}
where $p_\xi$ is as in \eqref{eq:p_xi_def} and $\xi$ is a positive scalar that will depend on the tolerances in \eqref{eq:sp_approx_sol} and \eqref{eq:dd_approx_sol}.
The above smoothing approximation scheme
is similar
to the one used in  \cite{nesterov2005smooth};
the approximation function $p_\xi$
used in both schemes 
is smooth,
but the one here is nonconvex while
the one in \cite{nesterov2005smooth}
is convex. Moreover, while
\cite{nesterov2005smooth} uses an ACG variant to approximately
solve \eqref{eq:smoothed_intro_prb}, the AIPP-S scheme uses the AIPP method discussed below for this purpose.

% It is worth mentioning that this smoothing scheme is similar
% to the one used in \cite{nesterov2005smooth}, but instead of using an ACG variant to solve \eqref{eq:smoothed_intro_prb} it uses the AIPP method.

\subsection{AIPP method for smooth CNO problems\label{subsec:aipp}}

This subsection describes the AIPP method studied in \cite{WJRproxmet1},
and its corresponding iteration complexity result, for solving a class
of smooth CNO problems.

We first describe the problem that the AIPP method is intended to
solve. Let ${\cal X}$ be a finite-dimensional inner product and consider
the smooth CNO problem 
\begin{equation}
\phi_{*}:=\inf_{x\in{\cal X}}\left[\phi(x):=f(x)+h(x)\right]\label{eq:nco_prob}
\end{equation}
where $h:{\cal {\cal X}}\mapsto(-\infty,\infty]$ and
function $f$ satisfy the following assumptions: 
\begin{itemize}
\item[(P1)] $h\in{\overline{\rm Conv}}({\cal X})$ and $f$ is differentiable on $\dom h$; 
\item[(P2)] for some $M\geq m>0$, the function $f$ satisfies 
\begin{align}
-\frac{m}{2}\|x'-x\|^{2} & \leq f(x')-\left[f(x)+\left\langle \nabla f(x),x'-x\right\rangle \right],\label{ineq:concavity_f}\\
\|\nabla f(x')-\nabla f(x)\| & \leq M\|x'-x\|,\label{ineq:Lipschitz_f}
\end{align}
for any $x,x'\in\dom h$; 
\item[(P3)] $\phi_{*}$ defined in \eqref{eq:nco_prob} is finite. 
\end{itemize}
We now make four remarks about the above assumptions.
First, it is well-known that a necessary condition for $x^{*}\in\dom h$
to be a local minimum of \eqref{eq:nco_prob} is that $x^{*}$ is
a stationary point of $\phi$, i.e. $0\in\nabla f(x^{*})+\pt h(x^{*})$.  Second, it is well-known that \eqref{ineq:Lipschitz_f} implies that \eqref{ineq:concavity_f} holds for any $m \in [-M,M]$. Third, it is easy to see from Proposition~\ref{prop:xi_facts} that $p_\xi$ in \eqref{eq:p_xi_def} satisfies assumption (P2) with $(M,f)=(L_\xi,p_\xi)$ where $L_\xi$ is as in \eqref{eq:y_xi_def}. Fourth, it is also easy to see that the function $p_\xi$ in \eqref{eq:p_xi_def} satisfies assumption (P3) with $\phi_* = \inf_{x\in X} \hat{p}_\xi(x)$ in view of assumption (A4) and Lemma~\ref{lem:smoothing_relation}.

For the purpose of discussing future complexity results,
we consider the following notion of an approximate stationary point of \eqref{eq:nco_prob}:
given a tolerance $\bar{\rho}>0$, a pair $(\bar{x},\bar{u})\in\dom h\times{\cal X}$
is said to be a $\bar{\rho}$--approximate stationary point of \eqref{eq:nco_prob}
if 
\begin{equation}
\bar{u}\in\nabla f(\bar{x})+\pt h(\bar{x}),\quad\|\bar{u}\|\leq\bar{\rho}.\label{eq:rho_approx_sol}
\end{equation}
We now state the AIPP method for finding a pair $(\bar x, \bar u)$ satisfying \eqref{eq:rho_approx_sol}.

\noindent \begin{minipage}[t]{1\columnwidth} 
\rule[0.5ex]{1\columnwidth}{1pt}

\noindent \textbf{AIPP method}

\noindent \rule[0.5ex]{1\columnwidth}{1pt}
\end{minipage}

\noindent \textbf{Input}: a function pair $(f,h)$, a scalar pair $(m,M)\in\r_{++}^{2}$ satisfying (P2), scalars $\lambda\in(0,1/(2m)]$ and  $\sigma\in(0,1)$, an initial point $x_{0}\in\dom h$, and
a tolerance $\bar{\rho}>0$; \vspace*{0.5em}

\noindent \textbf{Output}: a pair $(\bar{x},\bar{u})\in\dom h\times{\cal X}$
satisfying \eqref{eq:rho_approx_sol}; \vspace*{0.5em}

\begin{itemize}
\item[(0)] set $k=1$ and define $\hat{\rho}:=\bar{\rho}/{4}$, $\hat{\varepsilon}:=\bar{\rho}^{2}/[{32(M+\lam^{-1})}]$, and $M_{\lam}:=M+\lambda^{-1}$;
\item[(1)] call the accelerated
composite gradient (ACG) method in Appendix \ref{app:acg} with inputs 
$z_0=x_{k-1}$, $(\mu,L)=(1/2, \lam M + 1/2)$, $\psi_s = \lambda f+ \|\cdot-x_{k-1}\|^{2}/4$, and $\psi_n=\lambda h+\|\cdot-x_{k-1}\|^{2}/4$
in order to obtain a triple $(x,u,\varepsilon)\in{\cal X}\times{\cal X}\times\r_{+}$
satisfying 
\begin{equation}
u\in\partial_{\varepsilon}\left(\lambda\phi+\frac{1}{2}\|\cdot-x_{k-1}\|^{2}\right)(x),\quad\|u\|^{2}+2\varepsilon\leq\sigma\|x_{k-1}-x+u\|^{2};\label{eq:AIPPM_hpe}
\end{equation}
\item[(2)] if $\|x_{k-1}-x+u\|\leq\lambda\hat{\rho}/5$, 
then go to (3); otherwise set $(x_{k},\tilde{u}_{k},\tilde{\varepsilon}_{k})=(x,u,\varepsilon)$,
increment $k=k+1$ and go to (1); 
\item[(3)] restart the previous call to the ACG method in step 1 to find a triple
$(\tilde{x},\tilde{u},\tilde{\varepsilon})$ such that $\tilde{\varepsilon}\le\hat{\varepsilon}\lam$
and $(x,u,\varepsilon)=(\tilde{x},\tilde{u},\tilde{\varepsilon})$
satisfies \eqref{eq:AIPPM_hpe}; 
\item[(4)] compute 
\begin{align}
\bar{x} & :=\argmin_{x'\in{\cal X}}\left\{ \left\langle \nabla f(x),x'-x\right\rangle +h(x')+\frac{M_{\lambda}}{2}\|x'-x\|^{2}\right\} ,\label{eq:x_bar}\\
\bar{u} & :=M_{\lambda}(x-\bar{x})+\nabla f(\bar{x})-\nabla f(x),\label{eq:u_bar}
\end{align}
where $M_\lam$ is as in step~0, and output the pair $(\bar{x},\bar{u})$. 
\end{itemize}
\noindent \rule[0.5ex]{1\columnwidth}{1pt}

We now make four remarks about the above AIPP method. First, at the
$k^{{\rm th}}$ iteration of the method, its step 1 invokes an ACG method, whose description is given in Appendix
\ref{app:acg}, to approximately solve the strongly convex proximal
subproblem
\begin{equation}
\min_{x \in {\cal X}}\left\{ \lam\phi(x)+\frac{1}{2}\|x-x_{k-1}\|^{2}\right\} \label{eq:acg_prox_subprb}
\end{equation}
according to \eqref{eq:AIPPM_hpe}. Second, Lemma~\ref{lem:nest_complex} shows that  every ACG iterate $(z,u,\varepsilon)$
satisfies the inclusion in \eqref{eq:AIPPM_hpe}, and hence, only the inequality in \eqref{eq:AIPPM_hpe} needs to be verified. Third, note that
\eqref{ineq:Lipschitz_f} implies that the gradient of the function
$\psi_{s}$ defined in step 1 of the AIPP method is $(\lam M+1/2)$--Lipschitz
continuous. As a consequence, Lemma~\ref{lem:nest_complex} with
$L=\lam M+1/2$ implies that the triple $(z,u,\varepsilon)$ in step
1 of any iteration of the AIPP method can be obtained in ${\cal O}(\sqrt{[\lam M+1]/\sigma})$
ACG iterations. 

Note that the above method differs slightly from
the one presented in \cite{WJRproxmet1} in that it adds step~4 in
order to directly output a $\bar{\rho}$--approximate stationary point as
in \eqref{eq:rho_approx_sol}. The justification for the latter claim
follows from \cite[Lemma 12]{WJRproxmet1}, \cite[Theorem 13]{WJRproxmet1}, and \cite[Corollary 14]{WJRproxmet1}, which also imply
the following complexity result.

\begin{comment}
The iteration complexity of the AIPP method, in terms of the overall
number of ACG iterations, is given in the following result, whose
proof can be found in \cite[Corollary 14]{WJRproxmet1}. 
\end{comment}

\begin{prop}
\label{prop:aipp_rho_eps_compl}The AIPP method terminates with a
$\bar{\rho}$--approximate stationary point of \eqref{eq:nco_prob} in
\begin{equation}
{\cal O}\left(\sqrt{\lambda M+1}\left[\frac{R(\phi;\lambda)}{\sqrt{\sigma}(1-\sigma)^2\lambda^{2}\bar{\rho}^{2}}+\log_{1}^{+}(\lambda M)\right]\right)\label{eq:aipp_rho_comp}
\end{equation}
ACG iterations, where 
\begin{equation}
R(\phi;\lambda)=\inf_{x'}\left\{ \frac{1}{2}\|x_{0}-x'\|^{2}+\lambda\left[\phi(x')-\phi_{*}\right]\right\} .\label{eq:R_phi_lam}
\end{equation}
\end{prop}

Note that scaling $R(\phi;\lam)$ by $1/\lam$
and then shifting  by $\phi_*$
results in the $\lam$-Moreau envelope\footnote{See \cite[Chapter 1.G]{rockafellar2009variational} for an exact definition.} of $\phi$.
Moreover, $R(\phi;\lam)$
admits the upper bound
\begin{equation}
R(\phi;\lambda)\leq\min\left\{ \frac{1}{2}d_{0}^{2},\lambda\left[\phi(x_{0})-\phi_{*}\right]\right\} \label{eq:R_phi_lam_bd}
\end{equation}
where $d_{0}:=\inf\left\{ \|x_{0}-x_{*}\|:x_{*}\text{ is an optimal solution of }\eqref{eq:nco_prob}\right\} $.

\subsection{AIPP-S scheme for min-max CNO problems} \label{subsec:aipp_s}

We are now ready to state the AIPP-S scheme for finding approximate stationary points of the unconstrained min-max CNO problem \eqref{eq:intro_prb}.
% smoothing approximation method for finding
% approximate stationary points of \eqref{eq:P_opt_def}. 

It is stated
in a incomplete manner in the sense that it does not specify how the
 parameter $\xi$ and the tolerance $\rho$ used in its
step 2 are chosen. 
Two invocations of this method, with different
choices of $\xi$ and $\rho$, are considered in Propositions~\ref{prop:sp_aipp_facts}
and \ref{prop:dd_aipp_facts}, which describe the iteration complexities
for finding approximate stationary points  as in \eqref{eq:sp_approx_sol} and \eqref{eq:dd_approx_sol},
respectively.

\noindent \begin{minipage}[t]{1\columnwidth} 
\rule[0.5ex]{1\columnwidth}{1pt}

\noindent \textbf{AIPP-S scheme}

\noindent \rule[0.5ex]{1\columnwidth}{1pt}
\end{minipage}

\noindent \textbf{Input}: a triple $(m,L_{x},L_{y})\in\r_{++}^{3}$
satisfying (A3), 
a smoothing constant $\xi>0$, an initial point
$(x_{0},y_{0})\in X \times Y$, and a tolerance $\rho>0$; \vspace*{0.5em}

\noindent \textbf{Output}: a pair $(x,u)\in X\times{\cal X}$; \vspace*{0.5em}

\begin{itemize}
\item[(0)] set $L_{\xi}$ as in \eqref{eq:y_xi_def}, $\sigma=1/2$, $\lam=1/(4m)$, and define $p_{\xi}$ as
in \eqref{eq:p_xi_def}; 
\item[(1)] apply the AIPP method with inputs $(m,L_{\xi})$, $(p_{\xi},h)$,
$\lambda$, $\sigma$, $x_{0}$, and $\rho$ to obtain a pair $(x,u)$
satisfying 
\begin{equation}
u\in\nabla p_{\xi}(x)+\pt h(x),\quad\|u\|\leq\rho;\label{eq:approx_smoothed}
\end{equation}
\item[(2)] output the pair $(x,u)$. 
\end{itemize}
\rule[0.5ex]{1\columnwidth}{1pt}

We now give four remarks about the above method. First, the AIPP
method invoked in step 2 terminates due to \cite[Theorem 13]{WJRproxmet1} and the third and fourth remarks following assumptions (P1)--(P3).
Second, since the AIPP-S scheme is a one-pass method (as opposed to
an iterative method), the complexity of the AIPP-S scheme is essentially
that of the AIPP method. Third, similar to the smoothing scheme of
\cite{nesterov2005smooth} which assumes $m=0$, the AIPP-S scheme
is also a smoothing scheme for the case in which $m>0$. On the other
hand, in contrast to the algorithm of \cite{nesterov2005smooth} which
uses an ACG variant, AIPP-S invokes the AIPP method to solve \eqref{eq:smoothed_intro_prb}
due to its nonconvexity. Finally, while the AIPP method in step~2 is called with $(\sigma,\lam)=(1/2, 1/(4m))$, it can also be called with any $\sigma\in(0,1)$ and $\lam\in(0,1/(2m))$ to establish the desired termination of the AIPP-S scheme.

For the remainder
of this subsection, our goal will be to show that a careful selection
of the parameter $\xi$ and the tolerance $\rho$ will allow the AIPP-S
method to generate approximate stationary points as in \eqref{eq:dd_approx_sol} and \eqref{eq:sp_approx_sol}.

Before proceeding, we first present a bound on the quantity $R(\hat{p}_\xi;\lam)$ in terms of the data in problem \eqref{eq:intro_prb}. 
Its importance derives from the fact that the AIPP method applied
to the smoothed problem \eqref{eq:smoothed_intro_prb} yields the
bound \eqref{eq:aipp_rho_comp} with $\phi=\hat{p}_{\xi}$. 
% More specifically, a bound on $R(\hat{p}_\xi;\lam)$ in terms of $R(\hat{p};\lambda)$, and
% hence in terms of the data of our problem of interest in this subsection
% (see the proofs of Proposition~\ref{prop:dd_aipp_facts} and \ref{prop:sp_aipp_facts}).

% It is not clear how the pair output by the AIPP-S scheme is related
% to the definitions of approximate stationary points described in either
% \eqref{eq:dd_approx_sol} or \eqref{eq:sp_approx_sol}. 
\begin{lem}
\label{lem:R_bd}For every $\xi>0$ and $\lambda\geq0$, it holds that 
% \begin{equation}
% -\infty<\hat{p}(x)-\frac{D_{y}^{2}}{2\xi}\leq\hat{p}_{\xi}(x)\leq\hat{p}(x)\label{eq:p_xi_bd}
% \end{equation}
% as well as 
\begin{equation}
R(\hat{p}_{\xi};\lambda)\leq R(\hat{p};\lambda)+\frac{\lambda D_{y}^{2}}{2\xi},\label{eq:R_sp_bd}
\end{equation}
where $R(\cdot,\cdot)$ and $D_y$ are as in \eqref{eq:R_phi_lam} and \eqref{eq:Dy_def}, respectively.
\end{lem}

\begin{proof}
% (a) We first observe that for every $y_{0}\in Y$ we have
% \[
% \Phi(x,y)+h(x)\geq\Phi(x,y)-\frac{1}{2\xi}\|y-y_{0}\|^{2}+h(x)\geq\Phi(x,y)+h(x)-\frac{D_{y}^{2}}{2\xi}\quad\forall(x,y)\in Z.
% \]
% Hence, taking the supremum of the above quantities over $y\in Y$,
% using the definitions of $\hat{p},\hat{p}_{\xi},\Phi_{\xi},p_{\xi}$,
% and assumption (A5) gives 
% \[
% -\infty<\hat{p}(x)-\frac{D_{y}^{2}}{2\xi}\leq\hat{p}_{\xi}(x)\leq\hat{p}(x)\quad\forall x\in X
% \]
% which is the first set of inequalities. 
Using Lemma~\ref{lem:smoothing_relation} and the definitions of $\hat p$ and $\hat{p}_\xi$, it holds that
\begin{equation}
\hat{p}_{\xi}(x)-\inf_{x'}\hat{p}_{\xi}(x')\leq\hat{p}(x)-\inf_{x'}\hat{p}(x')+\frac{D_{y}^{2}}{2\xi},\quad\forall x\in X.\label{eq:primal_R}
\end{equation}
Multiplying the above expression by $(1-\sigma)\lambda$ and adding
the quantity $\|x_{0}-x\|^{2}/2$ yields the inequality 
\begin{align}
 & \frac{1}{2}\|x_{0}-x\|^{2}+(1-\sigma)\lambda\left[\hat{p}_{\xi}(x)-\inf_{x'}\hat{p}_{\xi}(x')\right]\nonumber \\
 & \leq\frac{1}{2}\|x_{0}-x\|^{2}+(1-\sigma)\lambda\left[\hat{p}(x)-\inf_{\tilde{x}}\hat{p}(x')\right]+(1-\sigma)\frac{\lambda D_{y}^{2}}{2\xi}\quad\forall x\in X,\label{eq:R_sp_pre_bd}
\end{align}
Taking the infimum of the above expression, and using the definition
of $R(\cdot;\cdot)$ in \eqref{eq:R_phi_lam} yields the desired conclusion. 
\end{proof}

We now show how the AIPP-S scheme
generates a $(\rho_{x},\rho_{y})$--primal-dual stationary point, i.e.
one satisfying \eqref{eq:sp_approx_sol}. Recall the definition of ``oracle call'' in the paragraph containing \eqref{eq:prox_oracles}.

\begin{prop}
\label{prop:sp_aipp_facts} For a given tolerance pair $(\rho_{x},\rho_{y})\in\r_{++}^{2}$,
let $(x,u)$ be the pair output by the AIPP-S scheme with input parameter
$\xi$ and tolerance $\rho$ satisfying 
$\xi\geq {D_{y}}/\rho_{y}$ and
$\rho=\rho_{x}$. 
Moreover, define 
\begin{equation}
(\bar{u},\bar{v}):=\left(u,\frac{y_{0}-y_{\xi}(x)}{\xi}\right),\quad(\bar{x},\bar{y}):=(x,y_{\xi}(x)),\label{eq:spec_output_sp_aipp}
\end{equation}
where $y_{\xi}$ is as in \eqref{eq:y_xi_def}. Then, the following
statements hold: 
\begin{itemize}
\item[(a)] the AIPP-S scheme performs 
\begin{equation}
{\cal O}\left(\Omega_{\xi}\left[\frac{m^2 R(\hat{p};1/(4m))}{\rho_{x}^{2}}+\frac{m D_{y}^{2}}{\xi\rho_{x}^{2}}+\log_{1}^{+}(\Omega_{\xi})\right]\right) \label{eq:sp_aipp_compl}
\end{equation}
oracle calls, where $R(\cdot;\cdot)$ and $D_y$ are as in \eqref{eq:R_phi_lam} and \eqref{eq:Dy_def}, respectively, and 
\begin{equation}
\Omega_{\xi}:= 1 + \frac{\sqrt{\xi}L_{y}+\sqrt{L_{x}}}{\sqrt{m}};
\label{eq:Omega_xi_def}
\end{equation}
\item[(b)] the quadruple $(\bar{u},\bar{v},\bar{x},\bar{y})$ is a $(\rho_{x},\rho_{y})$--primal-dual
stationary point of \eqref{eq:intro_prb}. 
\end{itemize}
\end{prop}

\begin{proof}
(a) Using the inequality in \eqref{eq:y_xi_def}, it holds that 
\begin{align}
\sqrt{\frac{L_{\xi}}{4m} + 1}\leq1+\sqrt{\frac{L_{\xi}}{4m}} \leq 1 + \frac{\sqrt{\xi}L_{y}+\sqrt{L_{x}}}{2\sqrt{m}}= \Theta(\Omega_{\xi}).\label{eq:L_xi_sqrt_bd}
\end{align}
Moreover, using Proposition~\ref{prop:aipp_rho_eps_compl} with $(\phi,M)=(\hat{p}_{\xi},L_{\xi})$, Lemma~\ref{lem:R_bd}, and bound \eqref{eq:L_xi_sqrt_bd}, it follows that the number of ACG iterations performed by the AIPP-S scheme is on the order given by \eqref{eq:sp_aipp_compl}. 
Since step 1 of the AIPP invokes
once the ACG variant in  Appendix~\ref{app:acg} with a pair $(\psi_s,\psi_n)$
of the form
\[
\psi_s =\lam p_\xi + \frac{1}{4}\|\cdot - \tilde{z}\|^2, \quad \psi_n = \lam h + \frac{1}{4}\|\cdot - \tilde{z}\|^2 
\]
for some $\tilde z$ and each
iteration of this ACG variant performs
${\cal O}(1)$ gradient evaluations
of $\psi_s$, ${\cal O}(1)$ function evaluations of $\psi_s$ and $\psi_n$, and 
${\cal O}(1)$ $\psi_n$-resolvent evaluations, it follows from
Proposition~\ref{prop:xi_facts}(b)
and the definition of an ``oracle call'' in the paragraph containing \eqref{eq:prox_oracles}
that each one of the above ACG iterations
requires ${\cal O}(1)$
oracle calls.
Statement (a) now follows from the
above two conclusions.

% The conclusion now follows from Proposition~\ref{prop:xi_facts}(b) and the fact that, for any $\tilde{x}\in {\cal X}$ and $\lam > 0$, the ACG method in Appendix~\ref{app:acg} applied with 
% \[
% \psi_s =\lam p_\xi + \frac{1}{4}\|\cdot - \tilde{x}\|^2, \quad \psi_n = \lam h + \frac{1}{4}\|\cdot - \tilde{x}\|^2, 
% \]
% performs ${\cal O}(1)$ oracle calls in each of its iterations.

(b) It follows from the definitions of $p_{\xi}$, tolerance $\rho$,
and $(\bar{y},\bar{u})$ in \eqref{eq:p_xi_def}, the choice of $\xi$ and $\rho$,
and \eqref{eq:spec_output_sp_aipp}, respectively, Proposition~\ref{prop:xi_facts}(b),
and the inclusion in \eqref{eq:approx_smoothed} that $\|\bar{u}\|\leq\rho_{x}$
and 
\[
\bar{u}\in\nabla p_{\xi}(\bar{x})+\pt h(\bar{x})=\nabla_{x}\Phi(\bar{x},y_{\xi}(\bar{x}))+\pt h(\bar{x})=\nabla_{x}\Phi(\bar{x},\bar{y})+\pt h(\bar{x}).
\]
Hence, we conclude that the top inclusion and the upper bound on $\|\bar{u}\|$
in \eqref{eq:sp_approx_sol} hold. Next, the optimality condition
of $\bar{y}=y_{\xi}(\bar{x})$ as a solution to \eqref{eq:p_xi_def}
and the definition of $\bar{v}$ in in \eqref{eq:p_xi_def} give 
\begin{equation}
0\in\pt\left[-\Phi(\bar{x},\cdot)\right](\bar{y})+\frac{\bar{y}-y_{0}}{\xi}=\pt\left[-\Phi(\bar{x},\cdot)\right](\bar{y})-\bar{v}\label{eq:v_incl}
\end{equation}
Moreover, the definition of $\xi$ implies that $\|\bar{v}\|={\|\bar{y}-y_{0}\|}/{\xi}\leq{D_{y}}/({D_{y}/\rho_{y}}) =\rho_{y}.\label{eq:v_bar_bd_prf}$
Hence, combining \eqref{eq:v_incl} and the previous identity,
we conclude that the bottom inclusion and the upper bound on $\|\bar{v}\|$
in \eqref{eq:sp_approx_sol} hold. 
\end{proof}
We now make three remarks about Proposition \ref{prop:sp_aipp_facts}.
First, recall that $R(\hat{p};1/(4m))$ in the complexity \eqref{eq:sp_aipp_compl}
can be majorized by the rightmost quantity in \eqref{eq:R_phi_lam_bd} with $(\phi,\lam)=(\hat p, 1/(4m))$.
Second, under the assumption that $\xi = D_y / \rho_y$, the complexity of AIPP-S scheme reduces
to

\begin{equation}
{\cal O}\left(m^{3/2}\cdot R(\hat{p}; 1/(4m))\cdot\left[\frac{L_{x}^{1/2}}{\rho_{x}^{2}}+\frac{L_{y}D_{y}^{1/2}}{\rho_{x}^{2}\rho_{y}^{1/2}}\right]\right)\label{eq:sp_compl_spec}
\end{equation}
under the reasonable assumption that the ${\cal O}(\rho_{x}^{-2}+\rho_{x}^{-2}\rho_{y}^{-1/2})$
term in \eqref{eq:sp_aipp_compl} dominates the other terms. Third,
recall from the last remark following the previous proposition that
when $Y$ is a singleton, \eqref{eq:intro_prb} becomes a special
instance of \eqref{eq:nco_prob} and the AIPP-S scheme becomes equivalent
to the AIPP method of Subsection~\ref{subsec:aipp}. It similarly
follows that the complexity in \eqref{eq:sp_compl_spec} reduces to
${\cal O}(\rho_{x}^{-2})$
and, in view of this remark, the ${\cal O}(\rho_{x}^{-2}\rho_{y}^{-1/2})$
term in \eqref{eq:sp_compl_spec} is attributed to the (possible)
nonsmoothness in \eqref{eq:intro_prb}.

We next show how the AIPP-S scheme
generates a point that is \emph{near} a $\delta$--directional stationary
point, i.e., one satisfying \eqref{eq:dd_approx_sol}. Recall the definition of ``oracle call'' in the paragraph containing \eqref{eq:prox_oracles}.

\begin{prop}
\label{prop:dd_aipp_facts}Let a tolerance pair $\delta > 0$ be given and
consider the AIPP-S scheme with input parameter $\xi$ and tolerance
$\rho$ satisfying 
$\xi\geq {D_y}/{\tau}$ and  $\rho={\delta}/{2}$ for some  $\tau \leq \min\left\{{m\delta^2}/{2D_y}, {\delta^2}/{32mD_y}\right\}$.
Then, the following statements hold: 
\begin{itemize}
\item[(a)] the AIPP-S scheme performs 
\begin{equation}
{\cal O}\left(\Omega_{\xi}\left[\frac{R(\hat{p};\lambda)}{\lambda^{2}\delta^{2}}+\frac{D_{y}^{2}}{\lambda\xi\delta^{2}}+\log_{1}^{+}(\Omega_{\xi})\right]\right)\label{eq:dd_aipp_comp}
\end{equation}
oracle calls where $\Omega_\xi$,  $R(\cdot;\cdot)$, and $D_y$ are as
in \eqref{eq:Omega_xi_def}, \eqref{eq:R_phi_lam}, and \eqref{eq:Dy_def}, respectively;
\item[(b)] the first argument $x$ in the
pair output by the AIPP-S scheme satisfies \eqref{eq:dd_approx_sol}. 
\end{itemize}
\end{prop}

\begin{proof}
(a) Using Proposition~\ref{prop:sp_aipp_facts} with $(\rho_x, \rho_y)=(\delta/2, \tau)$ and our assumption on $\tau$ it follows that the AIPP-S stops in a number of ACG iterations bounded above by  \eqref{eq:dd_aipp_comp}.

(b) Let $(x,u)$ be the $\bar\rho$--approximate stationary point of \eqref{eq:smoothed_intro_prb}
generated by the AIPP-S scheme (see step 2) under the given assumption on $\xi$ and $\bar{\rho}$. Defining $(\bar v, \bar y)$ as in \eqref{eq:spec_output_sp_aipp}, it follows from Proposition~\ref{prop:sp_aipp_facts} with $(\rho_x,\rho_y)=(\delta/2,\tau)$ that $(u, \bar v, x, \bar y)$ is a $(\delta/2, \tau)$--primal-dual stationary point of \eqref{eq:intro_prb}. As a consequence, it follows from Proposition~\ref{prop:impl1_statn} with $(\rho_x,\rho_y)=(\delta/2,\tau)$ that there exists a point $\hat x$ satisfying
\begin{align}
\|\hat x - x\| \leq \sqrt{\frac{2D_y\tau}{m}}, \quad 
\inf_{\|d\|\leq 1} \hat{p}'(\hat x;d) \geq - \frac{\delta}{2} - 2\sqrt{2m D_y \tau}.
\end{align}
Combining the above bounds with our assumption on $\tau$ yields the desired conclusion in view of \eqref{eq:dd_approx_sol}.
\end{proof}
We now give four remarks about the above result. First, recall that
$R(\hat{p};1/(4m))$ in the complexity \eqref{eq:dd_aipp_comp} can
be majorized by the rightmost quantity in \eqref{eq:R_phi_lam_bd} with $(\phi,\lam)=(\hat p, 1/(4m))$.
Second, Proposition~\ref{prop:dd_aipp_facts}(b) states that, while
$x$ not a stationary point itself, it is near a $\delta$--directional stationary
point $\hat{x}$. Third, under the assumption that the bounds on $\xi$ and $\tau$ in Proposition~\ref{prop:dd_aipp_facts} hold at equality, the complexity
of the AIPP-S scheme reduces to 
\begin{equation}
{\cal O}\left(m^{3/2}\cdot R(\hat{p};1/(4m))\cdot\left[\frac{L_{x}^{1/2}}{\delta^{2}}+\frac{L_{y}D_{y}}{\delta^{3}}\right]\right)\label{eq:dd_compl_spec}
\end{equation}
under the reasonable assumption that the ${\cal O}(\delta^{-2}+\delta^{-3})$
term in \eqref{eq:dd_aipp_comp} dominates the other ${\cal O}(\delta^{-1})$
terms. Fourth, when $Y$ is a singleton, it is easy to see that \eqref{eq:intro_prb}
becomes a special instance of \eqref{eq:nco_prob}, the AIPP-S scheme
becomes equivalent to the AIPP method of Subsection~\ref{subsec:aipp},
and the complexity in \eqref{eq:dd_compl_spec} reduces to 
${\cal O}(\delta^{-2})$.
In view of the last remark, the ${\cal O}(\delta^{-3})$ term in \eqref{eq:dd_compl_spec}
is attributed to the (possible) nonsmoothness in \eqref{eq:intro_prb}.

\section{Linearly-constrained min-max optimization\label{subsec:lc_mm_opt}}

This section presents our proposed QP-AIPP-S scheme for solving the linearly constrained min-max CNO problem \eqref{eq:intro_lin_prb}, and it is divided into two subsections. The first one reviews a QP-AIPP method for solving smooth linearly-constrained CNO problems. The second one presents the QP-AIPP-S scheme and its iteration complexity for finding stationary points as in \eqref{eq:lc_sp_approx_sol}. Throughout our presentation, we let ${\cal X}, {\cal Y},$ and ${\cal U}$ be finite dimensional inner product spaces.

Before proceeding, let us give the precise assumptions underlying the problem of interest and discuss the relevant notion of stationarity. For problem \eqref{eq:intro_lin_prb} suppose that assumptions (A0)--(A3) hold and that the linear operator $\cal A: {\cal X} \mapsto {\cal U}$ and vector $b \in \cal U$ satisfy:
\begin{itemize}
    \item[(A5)] ${\cal A}\not\equiv0$ and ${\cal F}:=\{x\in X:{\cal A}x=b\}\neq\emptyset$; 
    \item[(A6)] there exists $\hat{c}\geq0$ such that 
    $
    \inf_{x\in X} \left\{\hat p(x) + {\hat{c}}\|{\cal A} x - b\|^2 / 2\right\} >-\infty.
    $
\end{itemize}
Note that (A4) in Subsection~\ref{subsec:prelim_asmp} is replaced by (A6) which is required
by the QP-AIPP method of the next subsection. 

Analogous to the first remark following \eqref{eq:crit_conv_incl}, it is known that if $(x^*,y^*)$ satisfies \eqref{eq:saddle_point} for every $(x,y)\in{\cal F}\times Y$ and $\hat \Phi$ as in \eqref{eq:Phi_hat_def}, then there exists a multiplier $r^* \in {\cal U}$ such that
\begin{equation}
\left(\begin{array}{c}
0\\
0
\end{array}\right)\in\left(\begin{array}{c}
\nabla_{x}\Phi(x^{*},y^{*}) + A^*r^*\\
0
\end{array}\right)+\left(\begin{array}{c}
\partial h(x^{*})\\
\pt\left[-\Phi(x^{*},\cdot)\right](y^{*})
\end{array}\right),\label{eq:crit_constr_conv_incl}
\end{equation}
holds. Hence, in view of the third remark in the paragraph following \eqref{eq:ddir_conv}, we only consider the approximate version of \eqref{eq:crit_constr_conv_incl} which is \eqref{eq:lc_sp_approx_sol}.

We now briefly outline the idea of the QP-AIPP-S scheme. The main idea is to apply the QP-AIPP method described in the next subsection to the smooth linearly-constrained CNO problem
\begin{equation}
\min_{x\in X}\left\{p_{\xi}(x) + h(x):{\cal A}x=b\right\}, \label{eq:main_qp_sp_prb}
\end{equation}
where $p_\xi$ is as in \eqref{eq:p_def} and $\xi$ is a positive scalar that will depend on the tolerances in \eqref{eq:lc_sp_approx_sol}. This idea is similar to the one in Section~\ref{sec:mm_opt} in that it applies an accelerated solver to a perturbed version of the problem of interest.

\subsection{QP-AIPP method for constrained smooth CNO problems\label{subsec:aqp_aipp}}

This subsection describes the QP-AIPP method studied in \cite{WJRproxmet1},
and its corresponding iteration complexity, for solving linearly-constrained
smooth CNO problems.

We begin by describing the problem that the QP-AIPP method intends
to solve. Consider the linearly-constrained smooth CNO problem
\begin{equation}
\hat{\phi}_{*}:=\inf_{x \in {\cal X}}\{\phi(x):=f(x)+h(x):{\cal A}x=b\}\label{eq:nco_lc_prob}
\end{equation}
where $h:{\cal {\cal X}}\mapsto(-\infty,\infty]$ and a
function $f$ satisfy assumptions (P1)--(P3), the operator ${\cal A}:{\cal X}\mapsto{\cal U}$
is linear, $b\in{\cal U}$, and the following additional assumptions
hold: 
\begin{itemize}
\item[(Q1)] ${\cal A}\not\equiv0$ and ${\cal F}:=\{x\in\dom h:{\cal A}x=b\}\neq\emptyset$; 
\item[(Q2)] there exists $\hat{c}\geq0$ such that $\hat{\phi}_{\hat{c}}>-\infty$
where 
\begin{equation}
\hat{\phi}_{c}:=\inf_{x \in {\cal X}}\left\{ \phi_{c}(x):=\phi(x)+\frac{c}{2}\|{\cal A}x-b\|^{2}\right\} ,\quad\forall c\geq0.\label{eq:phi_hat_c_def}
\end{equation}
\end{itemize}

We now give some remarks about the above assumptions. First, similar to problem \eqref{eq:nco_prob}, it is well-known that a necessary
condition for $x^{*}\in\dom h$ to be a local minimum of \eqref{eq:nco_lc_prob}
is that $x^{*}$ satisfies $0\in\nabla f(x^{*})+\pt h(x^{*})+{\cal A}^{*}r^{*}$
for some $r^{*}\in{\cal U}$. Second, it is straightforward to verify that $(p, h, {\cal A}, b)$ in \eqref{eq:intro_lin_prb} satisfy (Q1)--(Q2) in view of assumptions (A5)--(A6). Third, since every feasible solution of \eqref{eq:nco_lc_prob} is also a feasible solution of \eqref{eq:phi_hat_c_def}, it follows from assumptions (Q2) that $\hat\phi_* \geq \hat\phi_{\hat c} > -\infty$. Fourth, if $\inf_{x\in{\cal X}} \phi(x) > -\infty$ (e.g., $\dom h$ is compact) then (Q2) holds with $\hat c = 0$.

Our interest in this subsection is in finding an  approximate stationary point
of \eqref{eq:nco_lc_prob} in the following sense: given a tolerance
pair $(\bar{\rho},\bar{\eta})\in\r_{++}^{2}$, a triple $(\bar{x},\bar{u},\bar{r})\in \dom h \times{\cal X}\times{\cal U}$
is said to be a $(\bar{\rho},\bar{\eta})$--approximate stationary point
of \eqref{eq:nco_lc_prob} if 
\begin{equation}
\bar{u}\in\nabla f(\bar{x})+\pt h(\bar{x})+{\cal A}^{*}\bar{r},\quad\|\bar{u}\|\leq\bar{\rho},\quad\|{\cal A}\bar{x}-b\|\leq\bar{\eta}.\label{eq:rho_eta_approx_sol}
\end{equation}

We now state the QP-AIPP method for finding $(\bar x, \bar u, \bar r)$ satisfying \eqref{eq:rho_eta_approx_sol}.

\noindent 
\begin{minipage}[t]{1\columnwidth} 
\rule[0.5ex]{1\columnwidth}{1pt}

\noindent \textbf{QP-AIPP method}

\noindent \rule[0.5ex]{1\columnwidth}{1pt}
\end{minipage}

\noindent \textbf{Input}: a function pair $(f,h)$, a scalar pair $(m,M)\in\r_{++}^{2}$ satisfying
\eqref{ineq:concavity_f}, scalars $\lambda\in(0,1/(2m)]$ and  $\sigma\in(0,1)$,
a scalar $\hat{c}$ satisfying assumption (Q2), an initial point $x_{0}\in\dom h$,
and a tolerance pair $(\bar{\rho},\bar{\eta})\in\r_{++}^{2}$; \vspace*{0.5em}

\noindent \textbf{Output}: a triple $(\bar{x},\bar{u},\bar{r})\in\dom h\times{\cal X}\times{\cal U}$
satisfying \eqref{eq:rho_eta_approx_sol}; \vspace*{0.5em}

\begin{itemize}
\item[(0)] set $c=\hat{c}+M/\|{\cal A}\|^{2}$;
\item[(1)] define the quantities 
\begin{equation}
M_{c}:=M+c\|{\cal A}\|^{2},\quad f_{c}:=f+\frac{c}{2}\|{\cal A}(\cdot)-b\|^{2},\quad\phi_{c}=f_{c}+h,\label{eq:M_c_f_c_def}
\end{equation}
and apply the AIPP method with inputs $(m,M_{c})$, $(f_{c},h)$,
$\lambda$, $\sigma$, $x_{0}$, and $\bar{\rho}$ to obtain a  $\bar{\rho}$--approximate
stationary point $(\bar{x},\bar{u})$ of \eqref{eq:nco_prob} with $f=f_c$; 
\item[(2)] if $\|{\cal A}\bar{x}-b\|>\bar{\eta}$ then set $c=2c$
and go to (1); otherwise, set $\bar{r}=c\left({\cal A}\bar{x}-b\right)$
and output the triple $(\bar{x},\bar{u},\bar{r})$.
\end{itemize}
\rule[0.5ex]{1\columnwidth}{1pt}

We now give two remarks about the above method. First, it straightforward
to see that QP-AIPP method terminates due to the results in \cite[Section 4]{WJRproxmet1}.
Second, in view of Proposition~\ref{prop:aipp_rho_eps_compl}
with $(\phi,M)=(\phi_{c},M_{c})$, it is easy to see that the number
of ACG iterations executed in step 1 at any iteration of the method
is 
\begin{equation}
{\cal O}\left(\sqrt{\lam M_{c} + 1}\left[\frac{R(\phi_{c};\lam)}{\sqrt{\sigma}(1-\sigma)^2\lam^2\bar{\rho}^{2}}+\log_{1}^{+}\left(\lam M_{c}\right)\right]\right)\label{eq:kth_aqp_compl}
\end{equation}
and that the pair $(\bar{x},\bar{u})$ computed in step~1 satisfies
the inclusion and the first inequality in \eqref{eq:rho_eta_approx_sol}.

We now focus on the iteration complexity of the QP-AIPP method. Before
proceeding, we first define the useful quantity 
\begin{equation}
R_{c}(\phi;\lambda):=\inf_{x'}\left\{ \frac{1}{2}\|x_{0}-x'\|^{2}+\lambda\left[\phi(x')-\hat{\phi}_{c}\right]:x'\in{\cal F}\right\} ,\label{eq:R_c_lam_def}
\end{equation}
for every $c\geq\hat{c}$, where $\phi_{c}$ is as defined in \eqref{eq:phi_hat_c_def}.
The quantity in \eqref{eq:R_c_lam_def} plays an analogous role as
\eqref{eq:R_phi_lam} in \eqref{eq:aipp_rho_comp} and, similar to the discussion following Proposition~\ref{prop:aipp_rho_eps_compl}, it is a scaled and shifted $\lam$-Moreau envelope of $\phi + \delta_{\cal F}$.
Moreover, due to \cite[Lemma 16]{WJRproxmet1},
it also admits the upper bound 
\begin{equation}
R_{c}(\phi;\lambda)\leq R_{\hat{c}}(\phi;\lambda)\leq\min\left\{ \frac{1}{2}\hat{d}_{0}^{2},\lambda\left[\hat{\phi}_{*}-\hat{\phi}_{\hat{c}}\right]\right\} \label{eq:R_c_lam_bd}
\end{equation}
where $\hat{\phi}_{*}$ is as defined in \eqref{eq:nco_lc_prob} and
\[
\hat{d}_{0}:=\inf\left\{ \|x_{0}-x_{*}\|:x_{*}\text{ is an optimal solution of }\eqref{eq:nco_lc_prob}\right\}.
\]

We now state the iteration complexity of the QP-AIPP method, whose
proof may be adapted from \cite[Lemma 12]{WJRproxmet1} and \cite[Theorem 18]{WJRproxmet1}. 
\begin{prop}
\label{prop:aqp_aipp_compl}Let a constant $\hat{c}$ as in assumption
(Q2), scalar $\sigma\in(0,1)$, curvature pair $(m,M)\in\r_{++}^{2}$,
and a tolerance pair $(\bar{\rho},\bar{\eta})\in\r_{+}^{2}$ be given.
Moreover, define
\begin{equation}
T_{\bar{\eta}}:=\frac{2 R_{\hat{c}}(\phi;\lam)}{\bar{\eta}^{2}(1-\sigma)\lam}+\hat{c},\quad\Theta_{\bar \eta}:= M+T_{\bar{\eta}}\|{\cal A}\|^{2} .\label{eq:Theta_def}
\end{equation}
Then, the QP-AIPP method outputs a triple $(\bar{x},\bar{u},\bar{r})$
satisfying \eqref{eq:rho_eta_approx_sol} in 
\begin{equation}
{\cal O}\left(\sqrt{\lam \Theta_{\bar \eta} + 1}\left[\frac{R_{\hat{c}}(\phi;\lam)}{\sqrt{\sigma} (1-\sigma)^2 \lam^2 \bar{\rho}^{2}}+\log_{1}^{+}\left(\lam \Theta_{\bar \eta}\right)\right]\right)\label{eq:aqp_aipp_compl}
\end{equation}
ACG iterations. 
\end{prop}

\subsection{QP-AIPP-S scheme for constrained min-max CNO problems\label{sec:qp_sp_aipp_solve}}

We are now ready to state the QP-AIPP smoothing scheme for finding an approximate primal-dual stationary
point of the linearly-constrained min-max CNO problem \eqref{eq:intro_lin_prb}.

\noindent 
\begin{minipage}[t]{1\columnwidth} 
\rule[0.5ex]{1\columnwidth}{1pt}

\noindent \textbf{QP-AIPP-S scheme}

\noindent \rule[0.5ex]{1\columnwidth}{1pt}
\end{minipage}

\noindent \textbf{Input}: a triple $(m,L_{x},L_{y})\in\r_{++}^{2}$
satisfying assumption (A3), a scalar $\hat{c}$
satisfying assumption (A6), a smoothing constant $\xi\geq D_{y}/\rho_{y}$,
an initial point $(x_{0},y_{0})\in X \times Y$, and a tolerance triple $(\rho_{x},\rho_{y},\eta)\in\r_{++}^{3}$; \vspace*{0.5em}

\noindent \textbf{Output}: a triple $(\bar{u},\bar{v},\bar{x},\bar{y},\bar{r})$
satisfying \eqref{eq:lc_sp_approx_sol}; \vspace*{0.5em}

\begin{itemize}
\item[(0)] set $L_{\xi}$ as in \eqref{eq:y_xi_def}, $\sigma=1/2$, $\lam=1/(4m)$, and define $p_{\xi}$ as
in \eqref{eq:p_xi_def}; 
\item[(1)] apply the QP-AIPP method of Subsection~\ref{subsec:aqp_aipp} with
inputs $(m,L_{\xi})$, $(p_{\xi},h)$, $\lam$, $\sigma$, $\hat{c}$, $x_{0}$,
and $(\rho_{x},\eta)$ to obtain a triple $(\bar{u},\bar{x},\bar{r})$
satisfying 
\begin{equation}
\bar{u}\in\nabla p_{\xi}(\bar{x})+\pt h(\bar{x})+A^{*}\bar{r},\quad\|\bar{u}\|\leq\rho_{x},\quad\|{\cal A}\bar{x}-b\|\leq\eta.\label{eq:qp_sp_aipp_incl}
\end{equation}
\item[(2)] define $(\bar{v},\bar{y})$ as in \eqref{eq:spec_output_sp_aipp}
and output the quintuple $(\bar{u},\bar{v},\bar{x},\bar{y},\bar{r})$. 
\end{itemize}
\rule[0.5ex]{1\columnwidth}{1pt}

Some remarks about the above method are in order. First, the
QP-AIPP method invoked in step 1 terminates due to the remarks following assumptions (Q1)--(Q2) and the results in
Subsection~\ref{subsec:aqp_aipp}. Second, since the QP-AIPP-S scheme
is a one-pass algorithm (as opposed to an iterative algorithm), the
complexity of the QP-AIPP-S scheme is essentially that of the QP-AIPP
method.  Finally, while the QP-AIPP method in step~2 is called with $(\sigma,\lam)=(1/2, 1/(4m))$, it can also be called with any $\sigma\in(0,1)$ and $\lam\in(0,1/(2m))$ to establish the desired termination of the QP-AIPP-S scheme.

We now show how the QP-AIPP-S scheme generates a point $(\bar u, \bar v, \bar x, \bar y, \bar r)$ satisfying \eqref{eq:lc_sp_approx_sol}. Recall the definition of ``oracle call'' in the paragraph containing \eqref{eq:prox_oracles}.

\begin{prop}
\label{prop:qp_sp_aipp_compl}Let a tolerance triple $(\rho_{x},\rho_{x},\eta)\in\r_{++}^{3}$
be given and let the quadruple $(\bar{u},\bar{v},\bar{x},\bar{y},\bar{r})$ be the
output obtained by the QP-AIPP-S scheme. Then the following properties
hold: 
\end{prop}

\begin{itemize}
\item[(a)] the QP-AIPP-S scheme terminates in 
\begin{equation}
{\cal O}\left(\Omega_{\xi,\eta}\left[\frac{m^2 R_{\hat{c}}(\hat{p};1/(4m))}{\rho_{x}^{2}}+\frac{m D_{y}^{2}}{\xi\rho_{x}^{2}}+\log_{1}^{+}\left(\Omega_{\xi,\eta}\right)\right]\right)\label{eq:comp_qp_sp_aipp}
\end{equation}
oracle calls, where 
\begin{equation}
\Omega_{\xi,\eta}:= \Omega_\xi +
\left(
R_{\hat{c}}(\hat{p};1/(4m)) + \frac{D_{y}^2}{m\xi}
\right)^{1/2} \frac{\|{\cal A}\|}{\eta} \label{eq:Theta_xi_eta_def}
\end{equation}
and $\Omega_\xi$, $R(\cdot;\cdot)$, and $D_y$ are as
in \eqref{eq:Omega_xi_def}, \eqref{eq:R_phi_lam}, and \eqref{eq:Dy_def}, respectively; 
\item[(b)] the quintuple $(\bar{u},\bar{v},\bar{x},\bar{y},\bar{r})$ satisfies \eqref{eq:lc_sp_approx_sol}. 
\end{itemize}
\begin{proof}
(a) Let $\Theta_{\eta}$ be as in \eqref{eq:Theta_def} with $M=L_\xi$. Using the same arguments as in Lemma~\ref{lem:R_bd}, it is easy
to see that  $R_{\hat{c}}(\hat{p}_{\xi};1/(4m))\leq R_{\hat{c}}(\hat{p};1/(4m))+ { D_{y}^{2}}/({8m\xi})$,
and hence, using \eqref{eq:L_xi_sqrt_bd}, we have
\begin{align}
\sqrt{\frac{\Theta_\eta}{4m} + 1} & \leq 1 + \sqrt{\frac{L_{\xi}}{4m}} + \sqrt{\frac{4 R_{\hat{c}}(\hat{p}_{\xi};1/(4m))\|{\cal A}\|^2}{\eta^2}} \nonumber \\
& \leq 1 + \frac{\sqrt{\xi}L_{y}+\sqrt{L_{x}}}{2\sqrt{m}} + 2\left(
R_{\hat{c}}(\hat{p};1/(4m)) + \frac{D_{y}^2}{8m\xi}
\right)^{1/2} \frac{\|{\cal A}\|}{\eta} = \Theta(\Omega_{\xi,\eta}) .\label{eq:Theta_sqrt_bd}
\end{align}
The complexity in \eqref{eq:comp_qp_sp_aipp} now follows from the above bound and
 Proposition~\ref{prop:aqp_aipp_compl} with $(\phi,f,M)=(p,p_{\xi},L_{\xi})$.

(b) The top inclusion and bounds involving $\|\bar u\|$ and $\|{\cal A}\bar x - b\|$ in \eqref{eq:lc_sp_approx_sol} follow from Proposition~\ref{prop:xi_facts}(b), the definition
of $\bar{y}$ in step 2 of the algorithm, and Proposition~\ref{prop:aqp_aipp_compl} with $f=p_\xi$. The bottom inclusion and bound involving $\|\bar v\|$ follow from similar arguments given for Proposition~\ref{prop:sp_aipp_facts}(b). 
\end{proof}
We now make three remarks about the above complexity bound. First,
recall that $R_{\hat{c}}(p;1/(4m))$ in the complexity \eqref{prop:qp_sp_aipp_compl}
can be majorized by the rightmost quantity in \eqref{eq:R_c_lam_bd} with $\lam=1/(4m)$.
Second, under the assumption that $\xi=D_{y}/\rho_{y}$,
the complexity of the QP-AIPP-S scheme reduces to 
\begin{equation}
{\cal O}\left(m^{3/2}\cdot R_{\hat{c}}(\hat{p};1/(4m))\cdot\left[\frac{L_{x}^{1/2}}{\rho_{x}^{2}}+\frac{L_{y}D_{y}^{1/2}}{\rho_{y}^{1/2}\rho_{x}^{2}}+\frac{m^{1/2}\|{\cal A}\| R_{\hat{c}}^{1/2}(p;1/(4m))}{\eta\rho_{x}^{2}}\right]\right),\label{eq:qp_sp_aipp_spec_compl}
\end{equation}
under the reasonable assumption that the ${\cal O}(\rho_{x}^{-2}+\eta^{-1}\rho_{x}^{-2}+\rho_{y}^{-1/2}\rho_{x}^{-2})$
term in \eqref{eq:comp_qp_sp_aipp} dominates the other terms. Third,
when $Y$ is a singleton, it is easy to see that \eqref{eq:intro_lin_prb}
becomes a special instance of the linearly-constrained smooth CNO
problem \eqref{eq:nco_lc_prob}, the QP-AIPP-S of this subsection
becomes equivalent to the QP-AIPP method of Subsection~\ref{subsec:aqp_aipp},
and the complexity in \eqref{eq:qp_sp_aipp_spec_compl} reduces to ${\cal O}(\eta^{-1}\rho_x^{-2})$. 
In view of the last remark, the ${\cal O}(\rho_{x}^{-2}\rho_{y}^{-1/2})$
term in \eqref{eq:qp_sp_aipp_spec_compl} is attributed to the (possible)
nonsmoothness in \eqref{eq:intro_lin_prb}.

Let us now conclude this section with a remark about the penalty subproblem
\begin{equation}
\min_{x\in X}\left\{p_{\xi}(x) + h(x) + \frac{c}{2}\|{\cal A}x-b\|^2\right\}, \label{eq:penalty_sp_prb}
\end{equation}
which is what the AIPP method considers every time it is called in the  QP-AIPP-S scheme (see step~1). First, observe that \eqref{eq:intro_lin_prb} can be equivalently reformulated as
\begin{equation}
\min_{x\in X} \max_{y\in Y, r\in {\cal U}} \left[\Psi(x,y,r):= \Phi(x,y)+h(x)+\inner{r}{{\cal A}x-b} \right]. \label{eq:Psi_def}
\end{equation}
Second, it is straightforward to verify that problem \eqref{eq:penalty_sp_prb} is equivalent to
\begin{equation}
\min_{x \in X}\left\{ \hat{p}_{c,\xi}(x):=p_{c,\xi}(x)+h(x)\right\} ,\label{eq:penalty_composite_prb}
\end{equation}
where the function $p_{c,\xi}:X\mapsto\r$ is given by $p_{c,\xi}(x):=\max_{y\in Y,r \in {\cal U}}\{ \Psi(x,y,r)-\|r\|^{2}/({2c})-\|y-y_{0}\|^{2}/({2\xi})\}$, for every $ x\in X$,
and $\Psi$ as in \eqref{eq:Psi_def}. As a consequence, problem \eqref{eq:penalty_composite_prb} is similar to
\eqref{eq:smoothed_intro_prb} in that a smooth approximate is used in place of the nonsmooth component of 
the underlying saddle function $\Psi$. 

On the other hand, observe
that we cannot directly apply the smoothing scheme developed in Subsection~\ref{subsec:aipp_s}
to \eqref{eq:penalty_composite_prb} as the set ${\cal U}$ is generally
unbounded. One approach that avoids this problem is to invoke the
AIPP method of Subsection~\ref{subsec:aipp} to solve a sequence
subproblems of the form in \eqref{eq:penalty_composite_prb} for increasing
values of $c$. However, in view of the equivalence of \eqref{eq:penalty_sp_prb}
and \eqref{eq:penalty_composite_prb}, this is exactly the approach
taken by the QP-AIPP-S scheme of this section.

\section{Numerical experiments\label{sec:numerical}}

This section presents numerical results that illustrate the computational
efficiency of the our proposed smoothing scheme. It contains three subsections. Each subsection presents computational results for a specific unconstrained
nonconvex min-max optimization problem class.

Each unconstrained problem considered in this section is of the form
in \eqref{eq:intro_prb} and is such that the computation of the function
$y_{\xi}$ in \eqref{eq:y_xi_def} is easy. Moreover,
for a given initial point $x_{0}\in X$, three algorithms are
run for each problem instance until a quadruple $(\bar{u},\bar{v},\bar{x},\bar{y})$
satisfying the inclusion of \eqref{eq:sp_approx_sol} and
\begin{align}
\frac{\|\bar{u}\|}{\|\nabla p_{\xi}(z_{0})\|+1}\leq\rho_{x},\quad\|\bar{v}\|\leq\rho_{y},
\label{eq:numerical_obj}
\end{align}
is obtained, where $\xi=D_{y}/\rho_{y}$.

% Even though the proposed smoothing scheme of
% this paper invokes the AIPP method for solving a perturbed version
% of \eqref{eq:intro_prb}, the computational results use the more computationally
% efficient R-AIPP method (in place of the AIPP method) in a modified
% smoothing scheme called the R-AIPP smoothing (R-AIPP-S) scheme.

We now describe the three nonconvex-concave min-max methods
that are being compared in this section, namely: (i) the R-AIPP-S
method; (ii) the accelerated gradient smoothing (AG-S) scheme; and
(iii) the projected gradient step framework (PGSF). Both the AG-S
and R-AIPP-S schemes are modifications of the AIPP-S scheme which,
instead of using the AIPP method in its step~1, use the AG method of \cite{nonconv_lan16} and R-AIPP
method of \cite{WJRVarLam2018}, respectively. The PGSF is a simplified variant of Algorithm
2 of \cite[Subsection 4.1]{2019Nouiehed_Lee} which explicitly evaluates
the argmax function $\alpha^{*}(\cdot)$ in \cite[Section 4]{2019Nouiehed_Lee}
instead of applying an ACG variant to estimate its evaluation.

Regarding the penalty solvers, the AG method is implemented as described in Algorithm 2 of \cite{nonconv_lan16} while the R-AIPP method follows the implementation described in \cite[Section 5.3]{kong2021thesis}.

Note that, like the AIPP method, the R-AIPP similarly:
(i) invokes at each of its (outer) iterations an ACG method to inexactly
solve the proximal subproblem \eqref{eq:acg_prox_subprb}; and (ii)
outputs a $\bar{\rho}$--approximate stationary point of \eqref{eq:nco_prob}.
However, the R-AIPP method is more computationally efficient due to three key practical improvements over the AIPP method, namely: 
(i) it allows the stepsize $\lam$ to be significantly larger than the $1/(2m)$ upper bound in the AIPP method using adaptive estimates of $m$; 
(ii) it uses a weaker ACG termination criterion compared to the one in \eqref{eq:AIPPM_hpe};
and (iii) it does not prespecify the minimum number of ACG iterations as the AIPP method does in its step~1.

We next state some additional details about the numerical experiments.
First, each algorithm is run with a time limit of 4000 seconds. Second, the bold numbers in each of the computational
tables in this section highlight the algorithm that performed the
most efficiently in terms of iteration count or total runtime. Moreover,
each of tables contain a column labeled $\hat{p}_{\xi}(\bar{x})$
that contains the smallest obtained value of the smoothed function
in \eqref{eq:smoothed_intro_prb}, across all of the tested algorithms.
Third, the description of $y_{\xi}$ and choice of the constants
$m,L_{x},$ and $L_{y}$ for each of the considered optimization problems
can be found in \cite[Appendix I]{kong2021thesis}. Fourth, $y_{0}$ is
chosen to be 0 for all of the experiments. Finally, all algorithms
described at the beginning of this section are implemented in MATLAB
2019a and are run on Linux 64-bit machines each containing Xeon E5520
processors and at least 8 GB of memory. 

Before proceeding, it is worth mentioning that the code
for generating the results of this section is available online\footnote{See the examples in \texttt{./examples/minmax/} from the GitHub repository
\href{https://github.com/wwkong/nc_opt/}{https://github.com/wwkong/nc\_opt/}.}.

\subsection{Maximum of a finite number of nonconvex quadratic forms}

\label{subsec:max_qmp}

This subsection presents computational results for a minmax quadratic
vector problem, which is based on a similar problem in \cite{WJRVarLam2018}.

We first describe the problem. Given a dimension triple $(n,l,k)\in\n^{3}$,
a set of parameters $\{(\alpha_{i},\beta_{i})\}_{i=1}^{k}\subseteq\r_{++}^{2}$,
a set of vectors $\{d_{i}\}_{i=1}^{k}\subseteq\r^{l}$, a set of diagonal
matrices $\{D_{i}\}_{i=1}^{k}\subseteq\r^{n\times n}$, and matrices
$\{C_{i}\}_{i=1}^{k}\subseteq\r^{l\times n}$ and $\{B_{i}\}_{i=1}^{k}\subseteq\r^{n\times n}$,
the problem of interest is the quadratic vector minmax (QVM) problem
\[
\min_{x\in\r^{n}}\max_{y\in\r^{k}}\left\{ \delta_{\Delta^{n}}(x)+\sum_{i=1}^{k}y_{i}g_{i}(x):y\in\Delta^{k}\right\} ,
\]
where, for every index $1\leq i\leq k$, integer $p\in\n$, and $x\in\rn$, we define
$g_{i}(x):=\alpha_{i}\|C_{i}x-d_{i}\|^{2}/2-\beta_{i}\|D_{i}B_{i}x\|^{2}/2$ and $\Delta^{p}:=\left\{ z\in\r_{+}^{p}:\sum_{i=1}^{p}z_{i}=1,z\geq0\right\} $.

We now describe the experiment parameters for the instances considered.
First, the dimensions are set to be $(n,l,k)=(200,10,5)$ and only
5.0\% of the entries of the submatrices $B_{i}$ and $C_{i}$ are
nonzero. Second, the entries of $B_{i},C_{i},$ and $d_{i}$ (resp.,
$D_{i}$) are generated by sampling from the uniform distribution
${\cal U}[0,1]$ (resp., ${\cal U}[1,1000]$). Third, the initial
starting point is $z_{0}=I_{n}/n$, where $I_{n}$ is the $n$--dimensional
identity matrix. Fourth, with respect to the termination criterion, the inputs, for every $(x,y)\in\rn\times\r^{k}$,
are $\Phi(x,y)=\sum_{i=1}^{k}y_{i}g_{i}(x)$, $h(x)=\delta_{\Delta^{n}}(x)$, $\rho_{x}=10^{-2}$,$\rho_{y}=10^{-1}$, and $Y=\Delta^{k}$.
Finally, each problem instance considered is based on a specific curvature
pair $(m,M)$ satisfying $m\leq M$, for which each
scalar pair $(\alpha_{i},\beta_{i})\in\r_{++}^{2}$ is selected so
that $M=\lambda_{\text{\ensuremath{\max}}}(\nabla^{2}g_{i})$ and $-m=\lambda_{\min}(\nabla^{2}g_{i})$.

We now present the results in Table~\ref{tab:qvm_prb}.

\begin{table}[ht]
\begin{centering}
\makebox[\textwidth][c]{%
\begin{tabular}{|>{\centering}m{0.5cm}>{\centering}m{0.5cm}|>{\centering}p{1.8cm}|>{\centering}p{1.4cm}>{\centering}p{1.2cm}>{\centering}p{1.2cm}|>{\centering}p{1.4cm}>{\centering}p{1.2cm}>{\centering}p{1.2cm}|}
\hline 
\multirow{2}{0.5cm}{\centering{}{\footnotesize{}$M$}} & \multirow{2}{0.5cm}{\centering{}{\footnotesize{}$m$}} & \multirow{2}{1.8cm}{\centering{}{\footnotesize{}$\hat{p}_{\xi}(\bar{x})$}} & \multicolumn{3}{c|}{{\small{}Iteration Count}} & \multicolumn{3}{c|}{{\small{}Runtime}}\tabularnewline
 &  &  & {\footnotesize{}R-AIPP-S} & {\footnotesize{}AG-S} & {\footnotesize{}PGSF} & {\footnotesize{}R-AIPP-S} & {\footnotesize{}AG} & {\footnotesize{}PGSF}\tabularnewline
\hline 
{\footnotesize{}$10^{0}$} & {\footnotesize{}$10^{0}$} & {\footnotesize{}2.85E-01} & \textbf{\footnotesize{}23} & {\footnotesize{}294} & {\footnotesize{}1591} & \textbf{\footnotesize{}0.66} & {\footnotesize{}5.72} & {\footnotesize{}22.60}\tabularnewline
{\footnotesize{}$10^{1}$} & {\footnotesize{}$10^{0}$} & {\footnotesize{}2.88E+00} & \textbf{\footnotesize{}86} & {\footnotesize{}1371} & {\footnotesize{}14815} & \textbf{\footnotesize{}1.37} & {\footnotesize{}25.96} & {\footnotesize{}209.62}\tabularnewline
{\footnotesize{}$10^{2}$} & {\footnotesize{}$10^{0}$} & {\footnotesize{}2.85E+01} & \textbf{\footnotesize{}217} & {\footnotesize{}6270} & {\footnotesize{}150493} & \textbf{\footnotesize{}3.35} & {\footnotesize{}118.32} & {\footnotesize{}2122.93}\tabularnewline
{\footnotesize{}$10^{3}$} & {\footnotesize{}$10^{0}$} & {\footnotesize{}2.85E+02} & \textbf{\footnotesize{}1417} & {\footnotesize{}28989} & {\footnotesize{}-} & \textbf{\footnotesize{}21.58} & {\footnotesize{}546.25} & {\footnotesize{}4000.00{*}}\tabularnewline
\hline 
\end{tabular}}
\par\end{centering}
\caption{Iteration counts and runtimes for QVM problems. \label{tab:qvm_prb}}
\end{table}
\subsection{Truncated robust regression}

This subsection presents computational results for the robust regression
problem in \cite{rafique2018non}.

It is worth mentioning that \cite{rafique2018non} also presents a
min-max algorithm for obtaining a stationary point as in \eqref{eq:numerical_obj}.
However, its iteration complexity, which is ${\cal O}(\rho^{-6})$
when $\rho=\rho_{x}=\rho_{y}$, is significantly worse than the other
algorithms considered in this section and, hence, we choose not to
include this algorithm in our benchmarks.

We now describe the problem. Given a dimension pair $(n,k)\in\n^{2}$,
a set of $n$ data points $\{(a_{j},b_{j})\}_{i=1}^{n}\subseteq\r^{k}\times\{1,-1\}$
and a parameter $\alpha>0$, the problem of interest is the truncated
robust regression (TRR) problem 
\[
\min_{x\in\r^{k}}\max_{y\in\r^{n}}\left\{ \sum_{j=1}^{n}y_{j}(\phi_{\alpha}\circ\ell_{j})(x):y\in\Delta^{n}\right\} 
\]
where $\Delta^{n}$ is as in Subsection~\ref{subsec:max_qmp} with $p=n$, $\phi_{\alpha}(t):=\alpha\log\left(1+t/{\alpha}\right)$, and $\ell_{j}(x):=\log\left(1+e^{-b_{j}\left\langle a_{j},x\right\rangle }\right)$, 
for every $(\alpha,t,x)\in\r_{++}\times\r_{++}\times\r^{k}$, 

We now describe the experiment parameters for the instances considered.
First, $\alpha$ is set to $10$ and the data points $\{(a_{i},b_{i})\}$
are taken from different datasets in the LIBSVM library\footnote{See https://www.csie.ntu.edu.tw/\textasciitilde cjlin/libsvmtools/datasets/binary.html.}
for which each problem instance is based off of (see the ``data name''
column in the table below, which corresponds to a particular LIBSVM
dataset). Second, the initial starting point is $z_{0}=0$. Third,
with respect to the termination criterion,
the inputs, for every $(x,y)\in\r^{k}\times\rn$, are 
$\Phi(x,y)=\sum_{j=1}^{n}y_{j}(\phi_{\alpha}\circ\ell_{j})(x)$,  $h(x)=0$, $\rho_{x}=10^{-5}$, $\rho_{y}=10^{-3}$, and $Y=\Delta^{n}$.

We now present the results in Table~\ref{tab:trr_prb}.

\begin{table}[ht]
\begin{centering}
\makebox[\textwidth][c]{%
\begin{tabular}{|>{\centering}m{1.8cm}|>{\centering}p{1.8cm}|>{\centering}p{1.4cm}>{\centering}p{1.2cm}>{\centering}p{1.2cm}|>{\centering}p{1.4cm}>{\centering}p{1.2cm}>{\centering}p{1.2cm}|}
\hline 
\multirow{2}{1.8cm}{\centering{}{\small{}data name}} & \multirow{2}{1.8cm}{\centering{}{\footnotesize{}$\hat{p}_{\xi}(\bar{x})$}} & \multicolumn{3}{c|}{{\small{}Iteration Count}} & \multicolumn{3}{c|}{{\small{}Runtime}}\tabularnewline
 &  & {\footnotesize{}R-AIPP-S} & {\footnotesize{}AG-S} & {\footnotesize{}PGSF} & {\footnotesize{}R-AIPP-S} & {\footnotesize{}AG} & {\footnotesize{}PGSF}\tabularnewline
\hline 
{\footnotesize{}heart} & {\footnotesize{}6.70E-01} & \textbf{\footnotesize{}425} & {\footnotesize{}1747} & {\footnotesize{}6409} & \textbf{\footnotesize{}6.37} & {\footnotesize{}15.54} & {\footnotesize{}32.76}\tabularnewline
{\footnotesize{}diabetes} & {\footnotesize{}6.70E-01} & \textbf{\footnotesize{}852} & {\footnotesize{}1642} & {\footnotesize{}3718} & \textbf{\footnotesize{}8.61} & {\footnotesize{}24.12} & {\footnotesize{}52.77}\tabularnewline
{\footnotesize{}ionosphere} & {\footnotesize{}6.70E-01} & \textbf{\footnotesize{}1197} & {\footnotesize{}8328} & {\footnotesize{}54481} & \textbf{\footnotesize{}8.26} & {\footnotesize{}63.82} & {\footnotesize{}320.72}\tabularnewline
{\footnotesize{}sonar} & {\footnotesize{}6.70E-01} & \textbf{\footnotesize{}45350} & {\footnotesize{}96209} & {\footnotesize{}-} & \textbf{\footnotesize{}461.52} & {\footnotesize{}580.37} & {\footnotesize{}4000.00{*}}\tabularnewline
{\footnotesize{}breast-cancer} & {\footnotesize{}1.11E-03} & \textbf{\footnotesize{}46097} & {\footnotesize{}-} & {\footnotesize{}-} & \textbf{\footnotesize{}476.59} & {\footnotesize{}4000.00{*}} & {\footnotesize{}4000.00{*}}\tabularnewline
\hline 
\end{tabular}}
\par\end{centering}
\caption{Iteration counts and runtimes for TRR problems \label{tab:trr_prb}}
\end{table}

\subsection{Power control in the presence of a jammer}

This subsection presents computational results for the power control
problem in \cite{Lu2019BlockAO}.

It is worth mentioning that \cite{Lu2019BlockAO} also presents a
min-max algorithm for obtaining stationary points for the aforementioned
problem. However, its termination criterion and notion of stationarity
are significantly different than what is being considered in this
paper and, hence, we choose not to include the algorithm of \cite{Lu2019BlockAO}
in our benchmarks.

We now describe the problem. Given a dimension pair $(N,K)\in\n^{2}$,
a pair of parameters $(\sigma,R)\in\r_{++}^{2}$, a 3D tensor ${\cal A}\in\r_{+}^{K\times K\times N}$,
and a matrix $B\in\r_{+}^{K\times N}$, the problem of interest is
the power control (PC) problem 
\[
\min_{X\in\r^{K\times N}}\max_{y\in\r^{N}}\left\{ \sum_{k=1}^{K}\sum_{n=1}^{N}f_{k,n}(X,y):0\leq X\leq R,0\leq y\leq\frac{N}{2},\right\} ,
\]
where, for every $(X,y)\in\r^{K\times N}\times\r^{N}$, 
\[
f_{k,n}(X,y):=-\log\left(1+\frac{{\cal A}_{k,k,n}X_{k,n}}{\sigma^{2}+B_{k,n}y_{n}+\sum_{j=1,j\neq k}^{K}{\cal A}_{j,k,n}X_{j,n}}\right).
\]

We now describe the experiment parameters for the instances considered.
First, the scalar parameters are set to be $(\sigma,R)=(1/\sqrt{2},K^{1/K})$
and the quantities ${\cal A}$ and $B$ are set to be the squared
moduli of the entries of two Gaussian sampled complex--valued matrices
${\cal H}\in\mathbb{C}^{K\times K\times N}$ and $P\in\mathbb{C}^{K\times N}$.
More precisely, the entries of ${\cal H}$ and $P$ are sampled from
the standard complex Gaussian distribution ${\cal CN}(0,1)$ with ${\cal A}_{j,k,n}=|{\cal H}_{j,k,n}|^{2}$ and  $B_{k,n}=|P_{k,n}|^{2}$ for every $(j,k,n)$.
Second, the initial starting point is $z_{0}=0$. Third, with respect
to the termination criterion, the inputs,
are  $\Phi(X,y)=\sum_{k=1}^{K}\sum_{n=1}^{N}f_{k,n}(X,y)$, $h(X)=\delta_{Q_{R}^{K\times N}}(X)$, $\rho_{x}=10^{-1}$, $\rho_{y}=10^{-1}$, and  $Y=Q_{N/2}^{N\times1}$,
for every $(X,y)\in\r^{K\times N}\times\r^{N}$ and $(U,V)\in\n^{2}$,
where $Q_{T}^{U\times V}:=\{z\in\r^{p\times q}:0\leq z\leq T\}$ for
every $T>0$. Fourth, each problem instance
considered is based on a specific dimension pair $(N,K)$. 

We now present the results in Table~\ref{tab:pc_prb}.

\begin{table}[ht]
\begin{centering}
\makebox[\textwidth][c]{%
\begin{tabular}{|>{\raggedright}m{0.3cm}>{\raggedright}m{0.3cm}|>{\centering}p{1.8cm}|>{\centering}p{1.4cm}>{\centering}p{1.2cm}>{\centering}p{1.2cm}|>{\centering}p{1.4cm}>{\centering}p{1.2cm}>{\centering}p{1.2cm}|}
\hline 
\multirow{2}{0.3cm}{\centering{}{\footnotesize{}$N$}} & \multirow{2}{0.3cm}{\centering{}{\footnotesize{}$K$}} & \multirow{2}{1.8cm}{\centering{}{\footnotesize{}$\hat{p}_{\xi}(\bar{x})$}} & \multicolumn{3}{c|}{{\small{}Iteration Count}} & \multicolumn{3}{c|}{{\small{}Runtime}}\tabularnewline
 &  &  & {\footnotesize{}R-AIPP-S} & {\footnotesize{}AG-S} & {\footnotesize{}PGSF} & {\footnotesize{}R-AIPP-S} & {\footnotesize{}AG} & {\footnotesize{}PGSF}\tabularnewline
\hline 
\centering{}{\footnotesize{}5} & \centering{}{\footnotesize{}5} & {\footnotesize{}-3.64E+00} & \textbf{\footnotesize{}37} & {\footnotesize{}322832} & {\footnotesize{}-} & \textbf{\footnotesize{}0.96} & {\footnotesize{}2371.27} & {\footnotesize{}4000.00{*}}\tabularnewline
\centering{}{\footnotesize{}10} & \centering{}{\footnotesize{}10} & {\footnotesize{}-2.82E+00} & \textbf{\footnotesize{}54} & {\footnotesize{}33399} & {\footnotesize{}-} & \textbf{\footnotesize{}0.75} & {\footnotesize{}293.60} & {\footnotesize{}4000.00{*}}\tabularnewline
\centering{}{\footnotesize{}25} & \centering{}{\footnotesize{}25} & {\footnotesize{}-4.52E+00} & \textbf{\footnotesize{}183} & {\footnotesize{}-} & {\footnotesize{}-} & \textbf{\footnotesize{}9.44} & {\footnotesize{}4000.00{*}} & {\footnotesize{}4000.00{*}}\tabularnewline
\centering{}{\footnotesize{}50} & \centering{}{\footnotesize{}50} & {\footnotesize{}-4.58E+00} & \textbf{\footnotesize{}566} & {\footnotesize{}-} & {\footnotesize{}-} & \textbf{\footnotesize{}40.89} & {\footnotesize{}4000.00{*}} & {\footnotesize{}4000.00{*}}\tabularnewline
\hline 
\end{tabular}}
\par\end{centering}
\caption{Iteration counts and runtimes for PC problems. \label{tab:pc_prb}}
\end{table}

\section{\label{sec:concl_remarks}Concluding Remarks}

This section makes some concluding remarks.

We first make a final remark about the AIPP-S smoothing scheme.
% Recall that section~\ref{sec:mm_opt} presents the AIPP-S
% scheme and analyzes its complexity with respect to two termination
% criteria, namely \eqref{eq:sp_approx_sol} and \eqref{eq:dd_approx_sol}. 
Recall that the
main idea of AIPP-S
is to call the AIPP method to obtain a pair satisfying
\eqref{eq:approx_smoothed}, or equivalently\footnote{See Lemma~\ref{lem:compl_approx2} with $f=p_\xi$.},
\begin{equation}
\inf_{\|d\|\leq1}(\hat{p}_{\xi})'(x;d)\geq-\rho.\label{eq:gen_aipp_term}
\end{equation}
 Moreover, using Proposition~\ref{prop:sp_aipp_facts} with $(\rho_x,\rho_y)=(\rho,D_y/\xi)$, it straightforward to see that that the number of oracle calls, in terms of $(\xi,\rho)$, is ${\cal O}(\rho^{-2} \xi^{1/2})$,
 which reduces
 to ${\cal O}(\rho^{-2.5})$
 if $\xi$ is chosen so as
 to satisfy
 $\xi=\Theta(\rho^{-1})$.
 The latter complexity bound
 improves upon the one
 obtained for an algorithm in
 \cite{2019Nouiehed_Lee}
which obtains a point $x$ satisfying
 \eqref{eq:gen_aipp_term} with $\xi=\Theta(\rho^{-1})$
in ${\cal O}(\rho^{-3})$ oracle calls. 
 
 We now discuss some possible
 extensions of this paper. First,
 it is worth investigating whether complexity results for the AIPP-S method can be derived for the case where $Y$ is unbounded.
 Second, it is worth investigating if the notions of stationary points in Subsection~\ref{subsec:prelim_asmp} are related to first-order
 stationary points\footnote{See, for example, \cite[Chapter~3]{luo1996mathematical}.} of the related mathematical program with equilibrium constraints:
 \[
 \min_{(x,y) \in X\times Y} \left\{ \Phi(x,y) + h(y) : 0 \in \pt [-\Phi(\cdot,y)](x) \right\}.
 \]
 Finally, it remains to be seen if a similar prox-type smoothing scheme can be developed for the case in which assumption (A2) is relaxed to the condition that there exists $m_y>0$ such that $-\Phi(x,\cdot)$ is $m_y$-weakly convex for every $x\in X$.

\noindent
 
\appendix
\appendixnotitle{\label{app:acg}}
This appendix contains a description and a result about an ACG variant
used in the analysis of \cite{WJRproxmet1}.

Part of the input of the ACG variant, which is described below, consists
of a pair of functions $(\psi_{s},\psi_{n})$ satisfying: 
\begin{itemize}
\item[(i)] $\psi_{n}\in\cConv$ is $\mu$--strongly convex for some $\mu\geq0$; 
\item[(ii)] $\psi_{s}$ is a convex differentiable function on $\dom\psi_{n}$
whose gradient is $L$--Lipschitz continuous for some $L>0$. 
\end{itemize}
\noindent \rule[0.5ex]{1\columnwidth}{1pt}

\noindent \textbf{ACG method}

\noindent \rule[0.5ex]{1\columnwidth}{1pt}

\noindent \textbf{Input}: a scalar pair $(\mu,L)\in\r_{++}^{2}$,
a function pair $(\psi_{n},\psi_{s})$, and an initial point $z_{0}\in\dom \psi_n$; \vspace*{0.5em}

\begin{itemize}
\item[(0)] set $y_{0}=z_{0}$, $A_{0}=0$, $\Gamma_{0}\equiv0$ and $j=0$; 
\item[(1)] compute 
\begin{align*}
A_{j+1} & =A_{j}+\frac{\mu A_{j}+1+\sqrt{(\mu A_{j}+1)^{2}+4L(\mu A_{j}+1)A_{j}}}{2L},\\
\tilde{z}_{j} & =\frac{A_{j}}{A_{j+1}}z_{j}+\frac{A_{j+1}-A_{j}}{A_{j+1}}y_{j},\\
\Gamma_{j+1}(y) & =\frac{A_{j}}{A_{j+1}}\Gamma_{j}(y)+\frac{A_{j+1}-A_{j}}{A_{j+1}}\left[\psi_{s}(\tilde{z}_{j})+\left\langle \nabla\psi_{s}(\tilde{z}_{j}),y-\tilde{z}_{j}\right\rangle \right] \quad \forall y, \\
y_{j+1} & =\argmin_{y}\left\{ \Gamma_{j+1}(y)+\psi_{n}(y)+\frac{1}{2A_{j+1}}\|y-y_{0}\|^{2}\right\} ,\\
z_{j+1} & =\frac{A_{j}}{A_{j+1}}z_{j}+\frac{A_{j+1}-A_{j}}{A_{j+1}}y_{j+1};
\end{align*}
\item[(2)] compute 
\begin{align*}
u_{j+1} & =\frac{y_{0}-y_{j+1}}{A_{j+1}},\\[2mm]
\varepsilon_{j+1} & =\psi(z_{j+1})-\Gamma_{j+1}(y_{j+1})-\psi_{n}(y_{j+1})-\langle u_{j+1},z_{j+1}-y_{j+1}\rangle;
\end{align*}
\item[(3)] increment $j=j+1$ and go to (1). 
\end{itemize}
\noindent \rule[0.5ex]{1\columnwidth}{1pt}

We now discuss some implementation details of the ACG method. First, a single iteration requires the
evaluation of two distinct types of oracles, namely: (i) the evaluation of
the functions $\psi_{n}$, $\psi_{s}$, $\nabla\psi_{s}$ at any point
in $\dom\psi_{n}$; and (ii) the computation of the exact solution
of subproblems of the form 
$\min_{y}\left\{ \psi_{n}(y)+ \|y-a\|^{2}/({2\alpha})\right\}$
for any $a\in{\cal Z}$ and $\alpha>0$. In particular, the latter is needed in the computation of $y_{j+1}$.
Second, because $\Gamma_{j+1}$ is affine, an efficient way
to store it is in terms of a normal vector and a scalar intercept
that is updated recursively at every iteration. Indeed, if $\Gamma_{j}=\alpha_{j}+\left\langle \cdot,\beta_{j}\right\rangle $
for some $(\alpha_{j},\beta_{j})\in\r\times{\cal Z}$, then step~1
of the ACG method implies that $\Gamma_{j+1}=\alpha_{j+1}+\left\langle \cdot,\beta_{j+1}\right\rangle $
where 
\begin{align*}
\alpha_{j+1} & :=\frac{A_{j}}{A_{j+1}}\alpha_{j}+\frac{A_{j+1}-A_{j}}{A_{j+1}}\left[\psi_{s}(\tilde{z}_{j})-\left\langle \nabla\psi_{j}(\tilde{z}_{j}),\tilde{z}_{j}\right\rangle \right],\\
\beta_{j+1} & :=\frac{A_{j}}{A_{j+1}}\beta_{j}+\frac{A_{j+1}-A_{j}}{A_{j+1}}\left[\nabla\psi_{s}(\tilde{z}_{j})\right].
\end{align*}

The following result, whose proof is given in \cite[Lemma 9]{WJRproxmet1},
is used to establish the iteration complexity of obtaining the triple
$(z,u,\varepsilon)$ in step 1 of the AIPP method of Subsection~\ref{subsec:aipp}. 
\begin{lem}
\label{lem:nest_complex} Let $\{(A_{j},z_{j},u_{j},\varepsilon_{j})\}$
be the sequence generated by the ACG method. Then, for any $\sigma>0$,
the ACG method obtains a triple $(z,u,\varepsilon)$ satisfying 
\begin{equation}
u\in\partial_{\varepsilon}(\psi_{s}+\psi_{n})(z)\quad\|u\|^{2}+2\varepsilon\le\sigma\|z_{0}-z+u\|^{2}\label{eq:acg_hpe}
\end{equation}
in at most $\left\lceil 2\sqrt{2L}(1+\sqrt{\sigma})/\sqrt{\sigma}\right\rceil$
iterations.
\end{lem}

\appendixnotitle{\label{app:smoothing}}
This appendix contains results about functions that can be described be 
as the maximum of a family of differentiable functions.

The technical lemma below, which is a special case of \cite[Theorem 10.2.1]{facchinei2007finite}, presents a key property about max functions.

\begin{lem} \label{lem:diff_danskin}
    Assume that the triple $(X,Y,\Psi)$ satisfies (A0)--(A1) in Subsection~\ref{subsec:prelim_asmp} with $\Phi=\Psi$. Moreover, 
    define
    \begin{equation}
    \quad q(x) := \sup_{y\in Y} \Psi(x, y), \quad Y(x) := \{y \in Y: \Psi(x, y)=q(x)\}, \quad  \forall x \in X. \label{eq:qYbar_def}
    \end{equation}
    Then, for every $(x,d)\in X\times {\cal X}$,
    it holds that 
    \[
        q'(x;d) = \max_{y \in Y(x)} \inner{\nabla_x \Psi(x; y)}{d}.
    \]
    Moreover, if ${Y(x)}$ reduces to a singleton, say ${Y(x)}=\{y(x)\}$, then $q$ is differentiable at $x$ and $\nabla q(x) = \nabla_x \Psi(x, y(x))$.
\end{lem}

Under assumptions (A0)--(A3) in Subsection~\ref{subsec:prelim_asmp}, the next result establishes Lipschitz continuity of the gradient of $q$. It is worth mentioning that it generalizes related results in \cite[Theorem 5.26]{beck2017first} (which covers the case where $\Psi$ is bilinear) and \cite[Proposition 4.1]{monteiro2010convergence} (which makes the stronger assumption that $\Psi(\cdot,y)$ is convex for every $y\in Y$).

\begin{prop}
\label{prop:Psi_global_ext}
Assume that the triple $(X, Y,\Psi)$
satisfies (A0)--(A3) in Subsection~\ref{subsec:prelim_asmp} with $\Phi=\Psi$ and that, for some $\mu>0$,
the function $\Psi(x,\cdot)$ is $\mu$-strongly concave on $Y$ for every $x\in X$, and define
\begin{align}
    Q_{\mu}:=\frac{L_{y}}{\mu}+\sqrt{\frac{L_{x}+m}{\mu}},
    \quad L_{\mu}:=L_{y}Q_{\mu}+L_{x}, \quad
    y(x) := \argmax_{y\in Y} \Psi(x, y)
    \label{eq:L_mu}
\end{align}
for every $x \in X$. Then, the following properties hold:
\begin{itemize}
\item[(a)] $y(\cdot)$
% given by
% \[
% y(x) := \argmax_{y\in Y} \Psi(x, y) \quad \forall x\in X
% \]
is $Q_{\mu}$--Lipschitz continuous on $X$;
% where 
% \begin{equation}
% Q_{\mu}:=\frac{L_{y}}{\mu}+\sqrt{\frac{L_{x}+m}{\mu}};\label{eq:Q_mu}
% \end{equation}
 \item[(b)] $\nabla q(\cdot)$ is $L_{\mu}$--Lipschitz continuous on $X$
 where $q$ is as in \eqref{eq:qYbar_def}.
% \begin{equation}
% L_{\mu}:=L_{y}Q_{\mu}+L_{x}.\label{eq:L_mu}
% \end{equation}
\end{itemize}
\end{prop}

\begin{proof}
(a) Let $x,\tilde{x}\in X$ be given and denote $(y,\tilde{y})=(y(x),y(\tilde{x}))$.
Define $\alpha(u):=\Psi(u,y)-\Psi(u,\tilde{y})$ for every $u\in X$, 
and observe that the optimality conditions of $y$ and $\tilde{y}$ imply
that $\alpha(x)\geq {\mu}\|y-\tilde{y}\|^{2} /2$ and $-\alpha(\tilde{x})\geq {\mu}\|y-\tilde{y}\|^{2}/2$.
Using the previous inequalities,
\eqref{eq:sp_lower_curv}, \eqref{eq:M_L_xy}, \eqref{eq:sp_upper_curv},
and the Cauchy-Schwarz inequality, we conclude that 
\begin{align*}
\mu\|y-\tilde{y}\|^{2}\leq\alpha(x)-\alpha(\tilde{x}) & \leq\left\langle \nabla_{x}\Psi(x,y)-\nabla_{x}\Psi(x,\tilde{y}),x-\tilde{x}\right\rangle +\frac{L_{x}+m}{2}\|x-\tilde{x}\|^{2}\\
 & \leq\|\nabla_{x}\Psi(x,y)-\nabla_{x}\Psi(x,\tilde{y})\|\cdot\|x-\tilde{x}\|+\frac{L_{x}+m}{2}\|x-\tilde{x}\|^{2}\\
 & \leq L_{y}\|y-\tilde{y}\|\cdot\|x-\tilde{x}\|+\frac{L_{x}+m}{2}\|x-\tilde{x}\|^{2}.
\end{align*}
Considering the above as a quadratic inequality in $\|\tilde{y}-y\|$
yields the bound 
\begin{align*}
\|y-\tilde{y}\| & \leq\frac{1}{2\mu}\left[L_{y}\|x-\tilde{x}\|+\sqrt{L_{y}^{2}\|x-\tilde{x}\|^{2}+4\mu(L_{x}+m)\|x-\tilde{x}\|^{2}}\right]\\
 & \leq\left[\frac{L_{y}}{\mu}+\sqrt{\frac{L_{x}+m}{\mu}}\right]\|x-\tilde{x}\|=Q_{\mu}\|x-\tilde{x}\|
\end{align*}
which is the conclusion of (a).

(b) Let $x,\tilde{x}\in X$ be given and denote $(y,\tilde{y})=(y(x),y(\tilde{x}))$.
Using part (a), Lemma~\ref{lem:diff_danskin}, and \eqref{eq:M_L_xy} we have that 
\begin{align*}
\|\nabla q(x)-\nabla q(\tilde{x})\| & =\|\nabla_{x}\Psi(x,y)-\nabla_{x}\Psi(\tilde{x},\tilde{y})\|\\
 & \leq\|\nabla_{x}\Psi(x,y)-\nabla_{x}\Psi(x,\tilde{y})\|+\|\nabla_{x}\Psi(x,\tilde{y})-\nabla_{x}\Psi(\tilde{x},\tilde{y})\|\\
 & \leq L_{y}\|y-\tilde{y}\|+L_{x}\|x-\tilde{x}\|\leq(L_{y}Q_{\mu}+L_{x})\|x-\tilde{x}\|=L_{\mu}\|x-\tilde{x}\|,
\end{align*}
which is the conclusion of (b). 
\end{proof}

% \appendixnotitle{\label{app:gen_ddir}}
\appendixnotitle{\label{app:cvx_facts}}
The main goal of this appendix is to prove
Propositions~\ref{prop:dd_sp_cvx} and \ref{prop:nco_refine},
which
are used in the proofs of
Propositions~\ref{prop:impl1_statn}, \ref{eq:prop_prox_statn}, and \ref{prop:ne_statn_pt} given
in Appendix~\ref{app:statn_notions}.

% The first technical lemma presents some general results about proper convex functions and nonempty closed convex sets.

% \begin{lem} \label{lem:gen_conv_tech}
% Let $\psi$ be a convex function and let $C\subseteq {\cal X}$ be a nonempty closed convex set. Then, the following statements hold:
% \begin{itemize}
%     \item[(a)] $\inf_{\|d\|\le 1} \sigma_C(d) = \left[- \min_{u \in C} \|u\| \right]$;
%     \item[(b)] if $C \cap \ri(\dom \psi) \neq \emptyset$, then $\inf_{x \in C} \cl \psi (x) = \inf_{x \in C} \psi(x) < \infty$.
% \end{itemize}
% \end{lem}

% \begin{proof}
% (a) See, for example, the proof of \cite[Lemma 5.1]{burke1991exact} with $g=0$.

% (b) This statement can be easily proved by
% using the definition of the closure of a convex function and \cite[Theorem 7.5]{Rockafellar70}. Its detailed proof can be found, for example, in \cite[Appendix F.1]{kong2021thesis}.
% \end{proof}

The following well-known result presents an important property about the directional derivative of a composite function $f + h$.

% --------------------------------------------------------------------------------
%% NOTE: The reference to Rockafellar's variational analysis book is based on the
%% referencing in the paper "STOCHASTIC MODEL-BASED MINIMIZATION OF
%% WEAKLY CONVEX FUNCTIONS" by Davis and Drusvyatskiy.
% --------------------------------------------------------------------------------

\begin{lem}
\label{lem:compl_approx2}Let $h:{\cal X} \mapsto (-\infty,\infty]$ be a proper convex function and let $f$ be a differentiable function on $\dom h$. Then, for any $x\in \dom h$, it holds that 
\begin{equation}
\inf_{\|d\|\leq1} (f + h)'(x;d) = \inf_{\|d\|\leq 1} \left[ \inner{\nabla f(x)}{d} +
\sigma_{\partial h(x)}(d) \right] = \ - \inf_{u\in \nabla f(x) + \pt h(x)} \|u\| .\label{eq:term_lee_gen}
\end{equation}
\end{lem}

 The proof of Lemma~\ref{lem:compl_approx2} can be found for example in \cite[Exercise 8.8(c)]{rockafellar2009variational}.
 An alternative and more direct proof is given in \cite[Lemma F.1.2]{kong2021thesis}. It is also worth mentioning that if we further assumed that $\dom h = {\cal X}$, then the above result would follow from \cite[Lemma 5.1]{burke1988identification}.

The next technical lemma, which can be found in \cite[Corollary 3.3]{sion1958general}, presents a well-known min-max identity.

\begin{lem} \label{lem:sion_minimax}
    Let a convex set $D \subseteq {\cal X}$ and compact convex set $Y\subseteq {\cal Y}$ be given. Moreover, let $\psi:D\times Y\mapsto \r$ be a function in which $\psi(\cdot,y)$ is convex lower semicontinuous for every $y\in Y$ and $\psi(d,\cdot)$ is concave upper semicontinuous for every $d\in D$. Then, 
    \[
    \inf_{d\in{\cal X}} \sup_{y\in {\cal Y}} \psi(d,y) = \sup_{y\in {\cal Y}} \inf_{d\in{\cal X}}  \psi(d,y).
    \]
\end{lem}

The next result establishes an identity similar to Lemma~\ref{lem:compl_approx2} but for the case where $f$ is a max function.

\begin{prop} \label{prop:dd_sp_cvx}
% Let $\Psi:{\cal X}\times{\cal Y}\mapsto(-\infty,\infty]$ and $h:{\cal X}\mapsto(-\infty,\infty]$
% be functions satisfying 
Assume the quadruple $(\Psi, h, X, Y)$ satisfies
assumptions (A0)--(A3) of Subsection~\ref{subsec:prelim_asmp}
with $\Phi=\Psi$. Moreover, suppose that $\Psi(\cdot,y)$ is convex
for every $y\in Y$, and let $q$ and ${Y}(\cdot)$ be as in Lemma~\ref{lem:diff_danskin}.
Then, for every $\bar{x}\in X$, it holds that 
\begin{equation}
\inf_{\|d\|\leq1}(q+h)'(\bar{x};d) = \ -\inf_{u \in Q(\bar{x})}\|u\| \label{eq:spec_dd_min}
\end{equation}
where
$
Q(\bar{x}) := \pt h(\bar{x}) + \bigcup_{y \in Y(\bar x)}
$.
Moreover, if $\pt h(\bar x)$ is nonempty, then the infimum on the right-hand side of \eqref{eq:spec_dd_min} is achieved.
\end{prop}

\begin{proof}
Let $\bar x \in X$ and define 
% \begin{equation}
% \Psi_y(x) := \Psi(x,y), \quad \psi(d,y) := (\Psi_y+h)'(\bar{x};d), \quad \forall (d,x,y)\in {\cal X}\times \Omega\times Y.\label{eq:pPsi_def}
% \end{equation}
\begin{equation}
\psi(d,y) := (\Psi_y+h)'(\bar{x};d), \quad \forall (d,x,y)\in {\cal X}\times \Omega\times Y.\label{eq:pPsi_def}
\end{equation}
% the functions
% \begin{align}
% \begin{gathered}
% \Psi_{y}(x):=\Phi(x,y)+\left\langle \bar{v},y\right\rangle + m\|x-\bar{x}\|^{2}, \quad q_{\bar{v}}(x):=\max_{y\in Y}\Psi_{y}(x), \\
% {Y}_{\bar v}(x) :=\left\{ \bar{y}\in Y:\Psi_{y}(x) = q_{\bar v}(x)  \right\}, 
% \end{gathered} \label{eq:q_vbar_def}
% \end{align}
We claim that $\psi$ in \eqref{eq:pPsi_def} satisfies the assumptions on $\psi$ in Lemma~\ref{lem:sion_minimax} with $Y= {Y}(\bar x)$ and $D$ given by 
\[
D := \left\{ d \in {\cal Z} : \|d \| \le 1 , d \in F_X(\bar x) \right\},
\]
where $F_X(\bar x) := \{t(x-\bar x) : x\in X, \, t \geq 0\}$ is the set of feasible directions at $\bar x$.
Before showing this claim, we use it to show that \eqref{eq:spec_dd_min} holds. First observe that (A1) and Lemma~\ref{lem:diff_danskin} imply that $q'(\bar x;d) = \sup_{y\in Y} \Psi_y'(\bar x;d)$ for every $d\in {\cal X}$.
% First, note that the above four properties of $\Psi_{y}'(\bar x;d)$ and assumption (A2) imply that the function $(d,y)\mapsto(\Psi_{y} + h)'(\bar x;d)$ satisfies the assumptions on $\psi$ in Lemma~\ref{lem:sion_minimax} with $Y= {Y}_{\bar v}(\bar x)$ and $D$ given by 
% \[
% D := \left\{ d \in {\cal Z} : \|d \| \le 1 , d \in F_X(\bar x) \right\}
% \]
% where $F_X(\bar x) := \{t(x-\bar x) : x\in X, t \geq 0\}$ is the set of feasible directions of $X$ at $\bar x$. Hence, 
Using then Lemma~\ref{lem:sion_minimax} with $Y=Y(\bar{x})$, Lemma~\ref{lem:compl_approx2} with $(f,x)=(\Psi_{\bar{y}},\bar{x})$ for every $\bar{y} \in {Y}(\bar{x})$, and the previous observation, we have that 
\begin{align}
& \inf_{\|d\|\leq 1}({q} + h)'(\bar{x};d) = 
\inf_{d \in D}({q} + h)'(\bar{x};d)
=
\inf_{d \in D} \sup_{y\in {Y}(\bar{x})}(\Psi_{y} + h)'(\bar{x};d)
  \nonumber \\
 & =\inf_{d \in D}\sup_{y\in {Y}(\bar{x})} \psi(d,y) = \sup_{y\in {Y}(\bar{x})}\inf_{d \in D} \psi(d,y) =
 \sup_{y\in {Y}(\bar{x})} \inf_{\|d\|\leq 1} (\Psi_{y} + h)'(\bar{x};d) \nonumber \\
& = \sup_{y\in {Y}(\bar{x})}\left[-\inf_{u\in\nabla_{x}\Phi(\bar{x},y)+\pt h(\bar{x})}\|u\|\right] = \left[-\inf_{u\in Q(\bar{x})}\|u\|\right]. \label{eq:partial_equiv_approx}
\end{align}
Let us now assume that $\pt h(\bar x)$ is nonempty, and hence, $Q(\bar x)$ is nonempty as well. 
Note that continuity of the function $\nabla_x \Psi(\bar{x},\cdot)$ from assumption (A1) and the compactness of ${Y}(\bar x)$ imply that $Q$ is closed. 
Moreover, since $\|u\|\geq 0$, it holds that any sequence $\{u_k\}_{k\geq 1}$ where $\lim_{k\to \infty}\|u_k\|=\inf_{u\in Q(\bar x)}\|u\|$ is bounded. 
Combining the previous two remarks with the Bolzano-Weierstrass Theorem, we conclude that $\inf_{u\in Q(\bar x)}\|u\| = \min_{u\in Q(\bar x)}\|u\|$, and hence \eqref{eq:spec_dd_min} holds.

To complete the proof, we now justify the above claim on $\psi$. First, for any given $y \in {Y}(\bar{x})$, it follows from 
\cite[Theorem 23.1]{Rockafellar70} with $f(\cdot)=\Psi_{y}(\cdot)$ and the definitions
of $q$ and $Y(\bar x)$ that
\begin{equation}
\psi(d,\bar{y}) = \Psi_{\bar{y}}'(\bar{x};d)=\inf_{t>0}\frac{\Psi_{y}(\bar{x}+td)-q(\bar{x})}{t}\quad\forall d\in{\cal X}.\label{eq:Psi_ddir}
\end{equation}
Since assumption (A2) implies that
$\Psi(\bar{x},\cdot)$ is upper semicontinuous and concave
on $Y$, it follows from \eqref{eq:Psi_ddir},
\cite[Theorem 5.5]{Rockafellar70},
and \cite[Theorem 9.4]{Rockafellar70} that $\psi(d,\cdot)$ is upper semicontinuous and concave on $Y$ for every $d\in{\cal X}$.
On the other hand, since $\Psi(\cdot, y)$ is assumed to be lower semicontinuous and convex
on $X$ for every $y\in Y$, it follows from \eqref{eq:Psi_ddir}, the fact that $\bar x\in\intr \Omega$, and \cite[Theorem 23.4]{Rockafellar70}, that $\psi(\cdot, y)$ is lower semicontinuous and convex on ${\cal X}$, and hence $D\subseteq {\cal X}$,
for every $y \in {Y}(\bar x)$.
\end{proof}

The last technical result is a specialization of the one given in \cite[Theorem 4.2.1]{Hiriart2}.

\begin{prop}
\label{prop:nco_refine} Let a proper closed function $\phi:{\cal X}\mapsto(-\infty,\infty]$
and assume that $\phi+\|\cdot\|^{2}/{2\lambda}$ is $\mu$-strongly
convex for some scalars $\mu,\lam>0$. %and tolerance $(\hat{\varepsilon},\hat{\rho})\in\r_{++}^{2}$
%be given and assume that for some $(\hat{\varepsilon},\hat{\rho})\in\r_{++}^{2}$
If a quadruple $(x^{-},x,u,\varepsilon)\in{\cal X}\times\dom\phi\times{\cal X}\times\r_{+}$
together with $\lam$ satisfy the inclusion
$
u\in\pt_{\varepsilon}\left(\phi+\|\cdot-x^{-}\|^{2}/[2\lam]\right)(x),
$
then the point $\hat{x}\in\dom\phi$ given by 
\begin{equation}
\hat{x}:=\argmin_{x'}\left\{ \phi_{\lam}(x'):=\phi(x')+\frac{1}{2\lam}\|x'-x^{-}\|^{2}-\left\langle u,x'\right\rangle \right\} \label{eq:x_hat_def}
\end{equation}
satisfies 
\begin{equation}
\inf_{\|d\|\leq1}\phi'(\hat{x};d)\geq- \frac{1}{\lambda}\|x^{-}-x+\lambda u\| - \sqrt{\frac{2\varepsilon}{\lam^2\mu}} ,\quad\|\hat{x}-x\|\leq\sqrt{\frac{2\varepsilon}{\mu}}.\label{eq:phi_refine_nonsmooth}
\end{equation}
\end{prop}

\begin{proof}
We first observe that the assumed inclusion implies that
$\phi_{\lam}(x')\geq\phi_{\lam}(x)-{\varepsilon}$
for every $x' \in X$. Using the previous inequality at $x'=\hat{x}$, the optimality
of $\hat{x}$, and the $\mu$--strong convexity of $\phi_{\lam}$, we have that 
${\mu}\|\hat{x}-x\|^{2}/2\leq\phi_{\lam}(x)-\phi_{\lam}(\hat{x})\leq{\varepsilon}$ 
from which we conclude that $\|\hat{x}-x\|\leq\sqrt{2{\varepsilon}/\mu}$, i.e., the second inequality in \eqref{eq:phi_refine_nonsmooth}.

To show the other inequality, let $n_\lam:=x^{-}-x+\lambda u$. Using the definition of $\phi_{\lam}$, the triangle inequality, and the previous bound on $\|\hat x - x\|$, we obtain
\begin{align}
0 & \leq\inf_{\|d\|\leq1}\phi_{\lam}'(\hat{x};d)=\inf_{\|d\|\leq1}\phi'(\hat{x};d)-\frac{1}{\lambda}\left\langle d, n_\lam\right\rangle \nonumber \\ 
& \leq \inf_{\|d\|\leq1}\phi'(\hat{x};d) + \frac{\|n_\lam\|}{\lambda} + \frac{\|x - \hat x\|}{\lam} 
\leq \inf_{\|d\|\leq1}\phi'(\hat{x};d) + \frac{\|n_\lam\|}{\lambda} + \sqrt{\frac{2\varepsilon}{\lam^2\mu}},\label{eq:tech_nco_dd1}
\end{align}
which clearly implies the first inequality in \eqref{eq:phi_refine_nonsmooth}.
\end{proof}

\appendixnotitle{\label{app:statn_notions}}
This appendix presents the proofs of Propositions~\ref{prop:impl1_statn}, \ref{eq:prop_prox_statn}, and \ref{prop:ne_statn_pt}.

% approximate solutions
% of a smooth composite function $\phi$ with lower curvature and the
% directional derivatives of $\phi$. 

The first technical result shows that an approximate primal-dual stationary point is equivalent to an approximate directional stationary point of a perturbed version of problem \eqref{eq:intro_prb}. 

% ------------------
% OLDER VERSION HERE
% ------------------
% \begin{lem} \label{lem:approx_unconstr}
% Let a pair $(\Phi,h)$ satisfying assumptions (A0)--(A3) of Subsection~\ref{subsec:prelim_asmp} and 
% $(\bar x,\bar{v})\in X \times {\cal Y}$ be given.
% Then,
% there exists $(\bar{u},\bar{y})\in {\cal X} \times Y$ such that
% $(\bar{u},\bar{v},\bar{x},\bar{y})$
% is a
% $(\rho_x,\rho_y)$--primal-dual stationary point of \eqref{eq:intro_prb} if and only if
% $\|\bar v\| \le \rho_y$ and $\bar{x}$ is a $\rho_x$-directional stationary point
% of the perturbed min-max problem
% \begin{equation}
% \min_{x\in X}\max_{y\in Y}\left[\Phi(x,y)+\left\langle \bar{v},y\right\rangle +h(x)\right]. \label{eq:perturbed_mm}
% \end{equation} 
% \[
% \inf_{\|d\| \le 1}
% p_{\bar u,\bar v}'(\bar x;d) \ge 0
% \]
% \end{lem}

\begin{lem} \label{lem:approx_unconstr}
Suppose the quadruple $(\Phi,h,X,Y)$ satisfies assumptions (A0)--(A3) of Subsection~\ref{subsec:prelim_asmp} and let
$(\bar x,\bar{u},\bar{v})\in X \times{\cal X} \times {\cal Y}$ be given.
Then,
there exists $\bar{y}\in Y$ such that the quadruple $(\bar{u},\bar{v},\bar{x},\bar{y})$ satisfies the inclusion in \eqref{eq:sp_approx_sol} if and only if
\begin{equation}
\inf_{\|d\| \le 1}
(p_{\bar{u}, \bar{v}}+h)'(\bar x;d) \ge 0, \label{eq:perturb_dd}
\end{equation}
where  $p_{_{\bar{u}, \bar{v}}} := \max_{y\in Y}  [\Phi(x,y)+\inner{\bar{v}}{y} - \inner{\bar u}{x}]$ for every $x\in \Omega$.
\end{lem}

\begin{proof}
Let $(\bar x, \bar u, \bar v)\in X\times {\cal X} \times {\cal Y}$ be given, define
\begin{align}
\Psi(x,y):=\Phi(x,y)+\inner{\bar{v}}{y} - \inner{\bar{u}}{x}+ m\|x-\bar{x}\|^{2} \quad \forall (x,y)\in \Omega\times Y, \label{eq:Psi_app_def}
\end{align}
and let $q$ and ${Y}(\cdot)$ be as in Lemma~\ref{lem:diff_danskin}. It is easy to see that $q=p_{\bar{u},\bar{v}}$, the function $\Psi$ satisfies the assumptions on $\Psi$ in Proposition~\ref{prop:dd_sp_cvx}, and $\bar x$ satisfies \eqref{eq:perturb_dd} if and only if $\inf_{\|d\|\leq 1} (q + h)'(\bar x;d) \geq 0$.
The desired conclusion follows from Proposition~\ref{prop:dd_sp_cvx}, the previous observation, and the fact that $\bar{y}\in {Y}(\bar{x})$ if and only if $\bar{v}\in\pt[-\Phi(\bar{x},\cdot)](\bar{y})$.
\end{proof}

We are now ready to give the proof of Proposition~\ref{prop:impl1_statn}.

\begin{proof}[Proof of Proposition~\ref{prop:impl1_statn}] Suppose $(\bar{u},\bar{v},\bar{x},\bar{y})$ is a $(\rho_x,\rho_y)$--primal-dual stationary point of \eqref{eq:intro_prb}. Moreover, let $\Psi$, ${q}$, and $D_y$ be as in \eqref{eq:Psi_app_def}, \eqref{eq:qYbar_def} and \eqref{eq:Dy_def}, respectively, and define 
\[
\hat q(x) := q(x) + h(x) \quad  \forall x\in X.
\]
Using Lemma~\ref{lem:approx_unconstr}, we first observe that $\inf_{\|d\|\leq 1} \hat{q}(\bar{x};d) \geq 0$. Since $\hat{q}$ is convex  from assumption (A3), it follows from the previous bound and  Lemma~\ref{lem:compl_approx2} with $(f,h)=(0, \hat{q})$, that $\min_{u\in \pt\hat{q}(\bar x)} \|u\| \leq  0$, and hence, $0\in\pt \hat{q}(\bar x)$. Moreover, using the Cauchy-Schwarz inequality, the second inequality in \eqref{eq:sp_approx_sol}, the previous inclusion, and the definition of $q$ and $\Psi$, it follows that for every $x \in {\cal X}$,
\begin{align*}
\hat{p}(x) + D_y\rho_y - \inner{\bar{u}}{x} + m\|x-\bar x\|^2 & \geq \hat{q}(x) \geq \hat{q}(\bar x) \geq \hat{p}(\bar x) - D_y \rho_y - \inner{\bar{u}}{\bar x},
\end{align*}
and hence that $\bar{u}\in \pt_\varepsilon (\hat{p} + m\|\cdot - \bar x\|^2)(\bar x)$ where $\varepsilon = 2D_y \rho_y$. Using now the first inequality in \eqref{eq:sp_approx_sol}, 
Proposition~\ref{prop:nco_refine} with $(\phi,x,x^-,u)=(\hat p, \bar x, \bar x,\bar{u})$ and also $(\varepsilon,\lam,\mu)=(D_y\rho_y, 1/(2m), m)$, we conclude that there exists $\hat x$ such that $\|\hat x - \bar x\| \leq \sqrt{2D_y\rho_y/m}$ and
\[
\inf_{\|d\|\leq 1} \hat{p}'(\hat{x};d) \geq -\|\bar{u}\| - 2\sqrt{2mD_y\rho_y} \geq -\rho_x - 2\sqrt{2mD_y\rho_y}.
\]

\end{proof}

We next give the proof of Proposition~\ref{eq:prop_prox_statn}.

\begin{proof}[Proof of Proposition~\ref{eq:prop_prox_statn}]

(a) We first claim that $\hat{P}_\lam$ is $\alpha$-strongly convex, where $\alpha = 1/\lambda - m$. To see this, note that $\Phi(\cdot,y)+m\|\cdot\|^2/2$ is convex for every $y\in Y$ from assumption (A3). The claim now follows from assumption (A2), the fact that the supremum of a collection of convex functions is also convex, and the definition of $\hat p$ in \eqref{eq:intro_prb}. 

Suppose the pair $(x,\delta)$ satisfies \eqref{eq:dd_approx_sol} and \eqref{eq:spec_delta_bd}. If $\hat x = x_\lambda$ in \eqref{eq:dd_approx_sol}, then clearly the second inequality in \eqref{eq:dd_approx_sol}, the fact that $\lam < 1/m$, and \eqref{eq:spec_delta_bd} imply the inequality in \eqref{eq:prox_stn_point}, and hence, that $x$ is a $(\lam, \varepsilon)$-prox stationary point. 
Suppose now that $\hat{x}\neq x_\lambda$. Using the convexity of
$\hat{P}_{\lambda}$, we first have that $\hat{P}'_{\lambda}(\hat{x};d)=\inf_{t>0}\left[\hat{P}_{\lambda}(\hat{x}+td) - \hat{P}_{\lambda}(\hat{x})\right]/t$
for every $d\in{\cal X}$. Denoting $n_\lam := (x_\lam-\hat{x})/{\|x_\lam-\hat{x}\|}$, using both inequalities in \eqref{eq:dd_approx_sol} and the previous identity, we then have that
\begin{align*}
\frac{\hat{P}_{\lambda}(x_\lam)-\hat{P}_{\lambda}(\hat{x})}{\|x_\lam-\hat{x}\|} & \geq \hat{p}'\left(\hat{x}; n_\lam\right) + \inner{\frac{n_\lam}{\lambda}}{\hat x - x} \geq-\delta - \frac{\|\hat x - x\|}{\lambda} \geq - \delta\left(\frac{1+\lambda}{\lambda} \right).
\end{align*}
Using the optimality of $x_{\lambda}$, the $\alpha$-strong convexity
of $\hat{P}_{\lambda}$ (see our claim on $\hat p$ in the first paragraph), and the above bound,
we conclude that 
\[
\frac{1}{2\alpha}\|\hat{x}-x_{\lambda}\|^{2}\leq \hat{P}_{\lambda}(\hat{x})-\hat{P}_{\lambda}(x_{\lambda})\leq \delta\left(\frac{1+\lambda}{\lambda} \right)\|\hat{x}-x_{\lambda}\|.
\]
Thus, $\|\hat{x}-x_{\lambda}\|\leq2\alpha\delta(1+\lambda)/\lambda$. Using the previous
bound, the second inequality in \eqref{eq:dd_approx_sol}, and \eqref{eq:spec_delta_bd} yields
\[
\|x-x_{\lambda}\|\leq\|x-\hat{x}\|+\|\hat{x}-x_{\lambda}\|\leq \left(1+2\alpha\left[\frac{1+\lambda}{\lambda}\right]\right)\delta\leq\lambda\varepsilon,
\]
which implies \eqref{eq:prox_stn_point}, and hence, that $x$ is a $(\lam, \varepsilon)$-prox stationary point. 

(b) Suppose
that the point $x$ is a $(\lambda,\varepsilon)$-prox stationary point with
$\varepsilon\leq\delta\cdot\min\{1,1/\lambda\}$. Then the optimality
of $x_{\lambda}$, the fact that $\hat{P}_{\lambda}$ is convex (see the beginning of part (a)), the inequality in \eqref{eq:prox_stn_point}, and the Cauchy-Schwarz
inequality imply that
\[
0\leq\inf_{\|d\|\leq1}\left[\hat{p}'(x_{\lambda};d)+\frac{1}{\lambda}\left\langle d,x_{\lambda}-x\right\rangle \right]\leq\inf_{\|d\|\leq1}\hat{p}'(x_{\lambda};d)+\varepsilon\leq\inf_{\|d\|\leq1}\hat{p}'(x_{\lambda};d)+\delta,
\]
which, together with the fact that $\lambda\varepsilon\leq\delta$, imply that $x$
satisfies \eqref{eq:dd_approx_sol} with $\hat{x}=x_{\lambda}$.
\end{proof}

Finally, we give the proof of Proposition~\ref{prop:ne_statn_pt}.

\begin{proof}[Proof of Proposition~\ref{prop:ne_statn_pt}]
This follows by using Lemma~\ref{lem:compl_approx2}  with $(f,h)=(\Phi(\cdot,\bar{y}),h)$ and $(f,h)=(0,-\Phi(\bar{x}, \cdot))$.
\end{proof}

% Using Lemma~\ref{lem:compl_approx2} with $\phi={\cal S}_x, -{\cal S}_y$, it is straightforward to see that if $(\bar u, \bar v, \bar x, y)$ satisfies \eqref{eq:sp_approx_sol} with $(\rho_x,\rho_y)=(\varepsilon_x,\varepsilon_y)$, then $(\bar x, y)$ is also a $(\varepsilon_x,\varepsilon_y)$-first-order Nash equilibrium point of \eqref{eq:intro_prb}. 

\bibliographystyle{siamplain}
\bibliography{Proxacc_ref}

\end{document}